\documentclass[12pt,a4paper]{article}

\usepackage{titleps,lipsum}
\newpagestyle{main}{%
  \setheadrule{0pt}%
  \sethead{\S\thesection~\sectiontitle}{}{\thepage}
}
\usepackage[nocompress]{cite}
\usepackage{listings}
\usepackage{docmute}
\usepackage{amssymb, enumerate, calligra}
\usepackage{amsthm}
\usepackage{amsmath}
\usepackage{amsxtra}
\usepackage{latexsym}
\usepackage{mathrsfs}
\usepackage[all,cmtip]{xy}
\usepackage[all]{xy}
\usepackage{enumitem}
\usepackage{xcolor}
\usepackage{graphicx}
\usepackage{comment}
\usepackage{mathabx,epsfig}

\usepackage[dvipdfmx]{hyperref}
\usepackage{cleveref}

\newtheorem{thm}{Theorem}[section]
\crefname{thm}{Theorem}{Theorems}
\newtheorem{lem}[thm]{Lemma}
\crefname{lem}{Lemma}{Lemmas}

\crefname{conj}{Conjecture}{Conjectures}
\newtheorem{claim}[thm]{Claim}
\crefname{claim}{Claim}{Claims}
\newtheorem{prop}[thm]{Proposition}
\crefname{prop}{Proposition}{Propositions}
\newtheorem{cor}[thm]{Corollary}
\crefname{cor}{Corollary}{Corollaries}

\crefname{property}{Property}{Properties}

\crefname{que}{Question}{Questions}

\theoremstyle{definition}
\newtheorem{defn}[thm]{Definition}
\crefname{defn}{Definition}{Definitions}
\newtheorem{rmk}[thm]{Remark}
\crefname{rmk}{Remark}{Remarks}

\numberwithin{equation}{section}

\crefname{ex}{Example}{Examples}
\newtheorem{cons}[thm]{Construction}
\crefname{cons}{Construction}{Constructions}

\def\Q{{\mathbb Q}}
\def\R{{\mathbb R}}
\def\Z{{\mathbb Z}}
\def\P{{\mathbb P}}
\def\A{{\mathbb A}}

\def\Hom{\mathop{\mathrm{Hom}}\nolimits}

\def\q{{ \mathfrak{q}}}     
\def\a{{ \mathfrak{a}}}     
\def\O{{ \mathcal{O}}}
\def\m{{ \mathfrak{m}}}  
\def\I{{ \mathcal{I}}}
\def\D{{ \mathcal{D}}}
\def\E{{ \mathcal{E}}}
\def\L{{ \mathcal{L}}}
\def\M{{ \mathcal{M}}}
\def\Y{{ \mathcal{Y}}}
\def\W{{ \mathcal{W}}}
\def\cR{{ \mathcal{R}}}
\def\K{{ \mathcal{K}}}
\def\cZ{{ \mathcal{Z}}}
\def\KK{ \overline{K}}
\def\Xbar{ \overline{X}}
\def\gl{ \gg \ll}
\def\tD{ \widetilde{D}}
\def\tY{ \widetilde{Y}}

\DeclareMathOperator{\pr}{pr}

\DeclareMathOperator{\id}{id}

\DeclareMathOperator{\Spec}{Spec}
\DeclareMathOperator{\Proj}{Proj}

\DeclareMathOperator{\Supp}{Supp}

\DeclareMathOperator{\ch}{char} 
\DeclareMathOperator{\Ass}{Ass} 
\DeclareMathOperator{\di}{div} 
\DeclareMathOperator{\Frac}{Frac} 

\DeclareFontFamily{U}{mathc}{}
\DeclareFontShape{U}{mathc}{m}{it}%
{<->s*[1.03] mathc10}{}

\DeclareMathAlphabet{\mathscr}{U}{mathc}{m}{it}

\DeclareMathOperator{\sHom}{\mathscr{Hom}}

\newenvironment{claimproof}[0]
  {%
   \paragraph{\it Proof.}%
  }
  {%
    \hfill$\blacksquare$%
  }

  {\end{list}}

\allowdisplaybreaks

\usepackage{blindtext}

\title{Height functions associated with closed subschemes}
\author{Yohsuke Matsuzawa\thanks{matsuzawa@math.brown.edu}
\thanks{Department of Mathematics, Box 1917, Brown University, Providence, Rhode Island 02912, USA}}

\begin{document}

\maketitle

\begin{abstract}
This is an expository note on height functions associated with closed subschemes.
This manuscript contains nothing new.
\end{abstract}

\setcounter{tocdepth}{2}
\tableofcontents

\section{Introduction}

This is an expository note on height function associated with closed subschemes with detailed proof.
Height functions associated with closed subschemes are introduced by Silverman in \cite{sil87}.
They are generalization of usual height functions associated to Cartier divisors.
Standard references for height functions associated with Cartier divisors are \cite{bg, Lan, hs}.

In most literatures, the base schemes on which we define height functions are assumed to be irreducible and reduced.
In this note, we do not assume these properties.

\section{Notation}
Let $k$ be a field.
\begin{enumerate}
\item A projective scheme over $k$ is a scheme over $k$ which has a closed immersion to the projective space $\P^{n}_{k}$ for some $n \in \Z_{>0}$ over $k$.
\item A quasi-projective scheme over $k$ is a scheme over $k$ which has an open immersion to a projective scheme over $k$.
\item An algebraic scheme over $k$ is a separated scheme of finite type over $k$.
\item Let $X$ be a scheme over $k$ and $k \subset k'$ be a field extension. The base change $X \times_{\Spec k}\Spec k'$ is denoted by $X_{k'}$.
For an ``object" $A$ on $X$, we sometimes use the notation $A_{k'}$ to express the base change of $A$ to $k'$ 
without mentioning to the definition of the base change if the meaning is clear.
\item For a Noetherian scheme $X$, the set of associated point of $X$ is denoted by $\Ass(X)$.
\end{enumerate}

\section{Absolute Values}

\subsection{Absolute values}
In this section, we recall basic facts on absolute values that we use to define height function later.
A good reference for absolute values is \cite[Chapter 9]{jacob}.

\begin{defn}
Let $K$ be a field.
An absolute value on $K$ is a map
\[
|\ |_{v} \colon K \longrightarrow \R_{\geq 0}; x \mapsto |x|_{v}
\]
which satisfies the following properties:
\begin{enumerate}
\item $|xy|_{v} = |x|_{v}|y|_{v}$ for all $x,y \in K$;
\item $|x+y|_{v}\leq |x|_{v}+|y|_{v}$ for all $x,y \in K$;
\item for $x\in K$, $|x|_{v}=0$ if and only if $x=0$.
\end{enumerate}
The absolute value $|\ |_{v}$ is called non-archimedean if the following holds instead of (2):
\[
|x+y|_{v}\leq \max\{|x|_{v}, |y|_{v}\} \quad \text{for all $x,y \in K$.}
\]
An absolute value which is not non-archimedean is called archimedean.
Two absolute values are said to be equivalent of they define the same topology on $K$.
\end{defn}

\begin{rmk}
An absolute value $|\ |_{v}$ is non-archimedean if and only if 
$|n\cdot 1|_{v} \leq 1$ for all $n \in \Z$ (the $1$ is the identity element of our field $K$).
\end{rmk}

We use the subscript $v$ to distinguish several absolute values.
We sometimes refer an absolute value as $v$ instead of $|\ |_{v}$ for simplicity.
In this note, even if we write simply $v$, this does not mean the place.

\begin{defn}
Let $K$ be a field.
Let $|\ |_{v}$ be a non-trivial absolute value on $K$.
We say $|\ |_{v}$ is well-behaved if for any finite field extension $K \subset L$, we have
\[
[L:K] = \sum_{w|v} [L_{w}:K_{v}] 
 \]
 where the sum runs over all absolute values $w$ on $L$ which extend $v$.
 (The notation $w|v$ means the restriction of $w$ on $K$ is $v$.)
 Here $L_{w}$ and $K_{v}$ are the completion of $L$ and $K$ with respect to $w, v$ respectively.
\end{defn}

\begin{rmk}
In the setting of the above definition,
the inequality $[L:K] \leq \sum_{w|v} [L_{w}:K_{v}] $ is always true.
\end{rmk}

\begin{rmk}
If $K$ is complete with respect to $v$, then $v$ is well-behaved because there is exactly one extension of $v$ to $L$
and $L$ is complete with respect to that absolute value.
\end{rmk}

\begin{rmk}
Let $K$ be a field of characteristic zero.
Then every non-trivial absolute value on $K$ is well-behaved.
\begin{proof}
Let $v$ be a non-trivial absolute value on $K$.
Let $L$ be a finite extension of $K$.
Then the extensions of $v$ to $L$ correspond to the prime ideals of $L {\otimes}_{K} K_{v}$.
Since $K$ has characteristic zero, we have
\[
L {\otimes}_{K} K_{v} \simeq L_{w_{1}} \times \cdots \times L_{w_{r}}
\]
where $w_{1}, \dots, w_{r}$ are the extensions of $v$ to $L$.
Thus we get
\[
[L:K] = \dim_{K_{v}} L {\otimes}_{K} K_{v} = \sum_{i=1}^{r} [L_{w_{i}}: K_{v}].
\]
\end{proof}
\end{rmk}

\begin{lem}
Let $K$ be a field and $v$ a well-behaved absolute value on $K$.
Then for any finite extension $K \subset L$ and any extension $w$ of $v$ to $L$,
$w$ is also a well-behaved absolute value. 
\end{lem}
\begin{proof}
Let $L'$ be a finite extension $L$.
Then we get
\begin{align*}
[L' : K] & = \sum_{\tiny \txt{ $u|v$\\ $u$ ab.\ val.\ on $L'$}}[L'_{u} : K_{v}] \qquad \text{since $v$ is well-behaved}\\[3mm]
& = \sum_{\tiny \txt{$w|v$\\ $w$ ab.\ val.\\on $L$}} \sum_{\tiny\txt{ $u|w$\\$u$ ab.\ val.\\ on $L'$}}[L'_{u}:L_{w}][L_{w}:K_{v}] \\[3mm]
& \leq [L':L] \sum_{\tiny \txt{$w|v$\\ $w$ ab.\ val.\\on $L$}}[L_{w}:K_{v}] = [L':L][L:K]=[L':K].
\end{align*}
Therefore, the inequality is actually equality and we are done.

\end{proof}

\begin{lem}\label{lem:nrm}
Let $K$ be a field and $v$ a well-behaved absolute value on $K$.
Let $K \subset L$ be a finite extension.
Then for any $ \alpha\in L$, we have
\[
\prod_{\tiny \txt{$w|v$, \\$w$ ab.\ val.\ on $L$}}| \alpha|_{w}^{[L_{w}:K_{v}]}=|N_{L/K}( \alpha)|_{v}.
\]
Here $N_{L/K}$ is the norm associated to the field extension.
\end{lem}
\begin{proof}
We have
\begin{align*}
N_{L/K}( \alpha) = {\det}_{K}(L \xrightarrow{ \alpha} L)={\det}_{K_{v}}(L {\otimes}_{K}K_{v} \xrightarrow{ \alpha {\otimes} 1} L {\otimes}_{K}K_{v} )
\end{align*}
where ${\det}_{K}$ (resp. ${\det}_{K_{v}}$) stand for determinant as $K$ (resp. $K_{v}$) linear maps. 
Since $v$ is well-behaved, we have
\[
L {\otimes}_{K} K_{v} \simeq \prod_{w|v} L_{w}.
\]
Thus we get
\begin{align*}
N_{L/K}( \alpha)=\prod_{w|v} {\det}_{K_{v}}(L_{w} \xrightarrow{ \alpha} L_{w}) = \prod_{w|v}N_{L_{w}/K_{v}}( \alpha).
\end{align*}
This shows
\[
|N_{L/K}( \alpha)|_{v} = \prod_{w|v}|N_{L_{w}/K_{v}}( \alpha)|_{v}=\prod_{w|v}| \alpha|_{w}^{[L_{w}:K_{v}]}.
\]

\end{proof}

\subsection{Proper set of absolute values}

\begin{defn}
Let $K$ be a field and $|\ |_{v}$ be an absolute value on $K$.
We say $  |\ |_{v}$ is proper if $  |\ |_{v}$ is non-trivial, well-behaved, and if $\ch K=0$,
then the restriction of $  |\ |_{v}$ to $\Q$ is equivalent to either trivial absolute value, $p$-adic absolute value,
or the usual absolute value.
\end{defn}

\begin{defn}\label{def:properset}
Let $K$ be a field.
A non-empty set $M_{K}$ of absolute values on $K$ is said to be proper if 
\begin{enumerate}
\item all element of $M_{K}$ is proper;
\item for all $  |\ |_{v},  |\ |_{w}\in M_{K}$, if $  |\ |_{v}\neq  |\ |_{w}$, then $  |\ |_{v}$ and $  |\ |_{w}$ are not equivalent;
\item for all $x\in K\setminus \{0\}$, $\sharp \left\{ v \in M_{K} \mid  |x|_{v}\neq1\right\}<\infty$.
\end{enumerate}
\end{defn}

\begin{rmk}
A proper set of absolute value $M_{K}$ contains only finitely many archimedean absolute values.
This is because the property (3) and archimedean absolute values restrict to the usual absolute value on $\Q$.
We denote the set of archimedean absolute values in $M_{K}$ by $M_{K}^{\infty}$.
The set of non-archimedean absolute values is denoted by $M_{K}^{\rm fin}$.
\end{rmk}

\begin{rmk}
Since finite fields do not have non-trivial absolute values, $K$ is infinite if it admits a proper set of absolute values.
\end{rmk}

\begin{defn}
Let $K$ be a field and $M_{K}$ be a proper set of absolute values on $K$.
For a field extension $K\subset L$, we set
\[
M_{K}(L) = M(L) =\left\{  |\ |_{w} \middle| \txt{$  |\ |_{w}$ is an absolute value on $L$\\ whose restriction to $K$ is an element of $M_{K}$}\right\}.
\]
Strictly speaking, this set depends on $M_{K}$ and the embedding $K \to L$.
But we usually fix $K$, $M_{K}$, and an algebraic closure $ \overline{K}$ of $K$ and work inside $ \overline{K}$, 
so there would not be any confusion.
\end{defn}

\begin{lem}
Let $K$ be a field and $M_{K}$ be a proper set of absolute values on $K$.
Let $K \subset L$ be a finite extension.
Then $M(L)$ is a proper set of absolute values on $L$.
\end{lem}
\begin{proof}
The only non-trivial thing is property (3) in \cref{def:properset}.
Let $ \alpha \in L\setminus \{0\}$.
We will show that $\sharp\{ w \in M(L) \mid | \alpha|_{w}\neq 1\}<\infty$. 
Let $f(X)=X^{n}+a_{n-1}X^{n-1}+\cdots + a_{0}$ be the minimal monic polynomial of $ \alpha$ over $K$.
Since $M_{K}$ is proper, there is a finite set $S \subset M_{K}$ containing $M_{K}^{\infty}$ such that
\begin{align*}
&|a_{i}|_{v}=1\ \text{or}\ 0\ \text{for $i=0, \dots, n-1$}\\
&|N_{L/K}( \alpha)|_{v}=1
\end{align*} 
for all $v\in M_{K}\setminus S$.
For any $v\in M_{K}\setminus S$, let $w \in M(L)$ be such that $w|v$.
Then $ \alpha$ (as an element of $L_{w}$) is integral over $\O_{v}=\{x\in K_{v} \mid |x|_{v}\leq 1\}$.
Thus $|N_{L_{w}/K_{v}}( \alpha)|_{v}\leq 1$.
By the proof of \cref{lem:nrm} and $|N_{L/K}( \alpha)|_{v}=1$, we get $|N_{L_{w}/K_{v}}( \alpha)|_{v} =1$ for all $w \in M(L)$
such that $w|v$.
Thus $| \alpha|_{w}=|N_{L_{w}/K_{v}}( \alpha)|_{v}^{1/[L_{w}:K_{v}]}=1$ for all $w \in M(L)$ such that $w|_{K}\in M_{K}\setminus S$.
\end{proof}

\section{Local heights}
In this section, we define local height functions associated to closed subschemes.
We will do this in several steps.

\subsection{Presentation of closed subschemes}
In this section, we work over an infinite field $k$.

\begin{lem}\label{lem:existpts}
Let $U\subset \A^{N}_{k}$ be a non-empty Zariski open subset.
Then $U(k) \neq \emptyset$.
\end{lem}
\begin{proof}
(sketch).

We may assume $U=D(f)(=\A^{N}_{k}\setminus (f=0))$, where $f \in k[x_{1},\dots,x_{N}]$ is a non zero polynomial.
It is enough to show that for any non zero polynomial $f \in k[x_{1},\dots,x_{N}]$, there is a point $a \in k^{N}$ such that $f(a) \neq 0$.
This can be shown by induction on $N$.
For $N=1$ case and proceeding induction, we need the assumption that $k$ is infinite.
\end{proof}

For an effective Cartier divisor $D$ on $X$, let $\O_{X}(D)= \sHom_{\O_{X}}(\I_{D}, \O_{X})$
where $\I_{D}$ is the ideal sheaf of $D$.
The global section of $\O_{X}(D)$ corresponding the inclusion $\I_{D} \to \O_{X}$ is denoted by $s_{D}$.

\begin{lem}\label{lem:existpres}
Let $X$ be a quasi-projective scheme over $k$.
\begin{enumerate}
\item For any effective Cartier divisor $D$ on $X$, there are globally generated invertible $ \O_{X}$-modules $ \mathcal{L}, \mathcal{M}$ such that 
$ \O_{X}(D) \simeq \mathcal{L} {\otimes} \mathcal{M}^{-1}$.

\item Let $ \mathcal{L}$ be a globally generated invertible $ \O_{X}$-module on $X$ and $x_{1},\dots x_{n} \in X $ be scheme points.
Then there are generating global sections $s_{0}, \dots, s_{n}$ of $ \mathcal{L}$ such that non of them vanish at any of $x_{i}$.

\item Let $Y \subset X$ be a closed subscheme such that $Y \cap \Ass(X)=\emptyset$.
Let $x_{1},\dots x_{n} \in X \setminus Y$ be scheme points.
Then there are effective Cartier divisors $D_{1}, \dots, D_{r}$ on $X$ such that
\begin{align*}
&Y=D_{1}\cap \cdots \cap D_{r} \qquad \text{as closed subschemes of $X$};\\
& x_{1}, \dots, x_{n} \notin D_{i} \qquad \text{for $i=1, \dots, r$}.
\end{align*}

\end{enumerate}
\end{lem}
\begin{proof}
(1) Since $X$ is quasi-projective, there is an ample invertible $\O_{X}$-module $\O_{X}(1)$.
We will write $\O_{X}(m)=\O_{X}(1)^{\otimes m}$.
For large enough $m>0$, $\O_{X}(m)$ and $\O_{X}(D) {\otimes} \O_{X}(m)$ are globally generated, and we are done.

(2)
Take a morphism to a projective space $\varphi \colon X \longrightarrow \P^{N}_{k}$ 
defined by a generating global sections of $ \mathcal{L}$.
Replace $X$, $ \mathcal{L}$, and $x_{1},\dots x_{n}$ with $\P^{N}_{k}$, $\O(1)$, and $\varphi(x_{1}),\dots , \varphi(x_{n})$,
we may assume $X=\P^{N}_{k}$ and $ \mathcal{L}=\O(1)$.
Non vanishing at $x_{1},\dots ,x_{n}$ is a Zariski open condition on $H^{0}(\O(1))=k^{N+1}$: let $U \subset H^{0}(\O(1))$ be the open set.
Generating $\O(1)$ is a Zariski open condition for $N+1$-tuples of elements of $H^{0}(\O(1))$: let $V \subset H^{0}(\O(1))^{N+1}$ be the open set.
Then by \cref{lem:existpts}, $U^{N+1}\cap V \neq \emptyset$ and we are done.

(3)
It follows from the following claim.
\begin{claim}
Let $X$ be a projective scheme and $Y \subset X$  a closed subscheme.
Let $x_{1},\dots, x_{n}\in X\setminus Y$ be scheme points.
Then there are locally principal closed subschemes $D_{1}, \dots, D_{r}$ of $X$ such that
$Y=D_{1}\cap \cdots \cap D_{r}$ as schemes and $x_{i} \notin D_{j}$ for all $i,j$.
\end{claim}
\begin{claimproof}
Let $ \mathcal{I}$ be the ideal of $Y \subset X$.
By replacing some specialization points, we may assume $x_{i}$ are closed points of $X$.
Let $\m_{x_{i}}$ be the ideal of the closed points $x_{i}$.
Let $\O_{X}(1)$ be an ample $\O_{X}$-module.
 
Let us consider the following exact sequence:
\[
0 \longrightarrow \m_{x_{1}}\cdots \m_{x_{n}} \longrightarrow \O_{X} \longrightarrow \bigoplus_{i=1}^{n}k(x_{i}) \longrightarrow0.
\]
Since $Y \cap\{x_{1},\dots, x_{n}\}=\emptyset$, we can tensor $\I$ preserving exactness:
\[
0 \longrightarrow \m_{x_{1}}\cdots \m_{x_{n}} \otimes \I \longrightarrow \I \longrightarrow \bigoplus_{i=1}^{n}k(x_{i}) \longrightarrow0.
\]
Let us take $m>0$ so that $\I \otimes \O_{X}(m)$ is globally generated and
$H^{1}(X, \m_{x_{1}}\cdots \m_{x_{n}} \otimes \I \otimes \O_{X}(m))=0$.
Then we get a surjection:
\[
H^{0}(X, \I \otimes \O_{X}(m)) \longrightarrow \bigoplus_{i=1}^{n}k(x_{i}).
\]
Take $k$-bases of $k(x_{i})$ and fix a $k$-vector isomorphism $\bigoplus_{i=1}^{n}k(x_{i}) \simeq k^{d}$.
Write the surjective $k$-linear map defined by the above map $\pi \colon H^{0}(X, \I \otimes \O_{X}(m)) \longrightarrow k^{d}$.
Since$k$ is infinite, by \cref{lem:existpts}, there is a $k$-basis $e_{1},\dots ,e_{d}$ of $k^{d}$ whose all coordinates are not zero.
Let $s_{1},\dots, s_{d} \in H^{0}(X, \I \otimes \O_{X}(m))$ be lifts of $e_{i}$: $\pi(s_{i})=e_{i}$.
Extend this to generating sections $s_{1},\dots, s_{d},s'_{d+1},\dots, s'_{r}$.
Since $e_{1},\dots, e_{d}$ is a $k$-basis of $k^{d}$, there are $a_{ij}\in k$ such that
all coordinates of
\[
\pi(s'_{i}+\sum_{j=1}^{d}a_{ij}s_{j}) \quad i=d+1,\dots,r
\]
are not zero.
Set $s_{i}=s'_{i}+\sum_{j=1}^{d}a_{ij}s_{j}$ for $ i=d+1,\dots,r$.
These $s_{1},\dots ,s_{r}$ as elements of $H^{0}(X, \O_{X}(m))$ define locally principal closed subschemes $D_{1},\dots, D_{r}$ of $X$.
By the construction, $Y=D_{1}\cap \cdots \cap D_{r}$ and $x_{i}\notin D_{j}$.
\end{claimproof}

\end{proof}

\begin{defn}[\bf Presentations of closed subschemes]
Let $X$ be a quasi-projective scheme over $k$.
\begin{enumerate}
\item
Let $ \mathcal{L}$ be an invertible $\O_{X}$-module. A global section $s \in H^{0}(X, \mathcal{L})$ is called a regular global section 
if the homomorphism $\O_{X} \longrightarrow \mathcal{L}$ defined by $s$ is injective.
Note that this is equivalent to say that $s$ does not vanish at any of associated points of $X$.
\item
Let $D$ be an effective Cartier divisor on $X$. A presentation of $D$ is
\[
\D = (s_{D}; \L, s_{0}, \dots, s_{n};\M, t_{0},\dots, t_{m}; \psi)
\]
where
$\L$, $\M$ are globally generated invertible $\O_{X}$-modules, $\psi \colon \L \otimes \M^{-1} \to \O_{X}(D) $ is an isomorphism,
and $s_{0},\dots, s_{n} \in H^{0}(X, \L)$ and $t_{0},\dots, t_{m} \in H^{0}(X, \M)$ are regular global sections which generate $\L$ and $\M$.
We usually omit $\psi$ and write 
\[
\D = (s_{D}; \L, s_{0}, \dots, s_{n};\M, t_{0},\dots, t_{m})
\]
for simplicity.
\item
Let $Y\subset X$ be a closed subscheme such that $Y\cap \Ass(X)=\emptyset$.
A presentation of $Y$ is 
\[
\Y =(Y; \D_{1},\dots, \D_{r})
\]
where $\D_{i}$ are presentations of effective Cartier divisors $D_{i}$ on $X$ such that
$Y=D_{1}\cap \cdots, \cap D_{r}$ as closed subschemes.
\end{enumerate}
\end{defn}

\begin{rmk}
By \cref{lem:existpres}, those presentations are always exist.
\end{rmk}

\begin{defn}
Let $X$ be a quasi-projective scheme over $k$.
Let $D, E$ be two effective Cartier divisors on $X$ and 
\begin{align*}
&\D = (s_{D}; \L, s_{0}, \dots, s_{n};\M, t_{0},\dots, t_{m}; \psi);\\
&\E = (s_{E}; \L', s'_{0}, \dots, s'_{n'};\M', t'_{0},\dots, t'_{m'}; \psi')
\end{align*}
presentations of them.
\begin{enumerate}
\item
We define the sum of $\D$ and $\E$ as
\begin{align*}
&\D+\E = \\
& \Bigl( s_{D+E};  \L \otimes \L', \{s_{i} \otimes s'_{j}\}_{\tiny \txt{$0\leq i \leq n$\\  $0\leq j \leq n'$}} 
; \M\otimes \M', \{t_{i}\otimes t'_{j}\}_{\tiny \txt{$0\leq i \leq m$\\$0\leq j \leq m'$}} ; \psi \otimes \psi' \Bigr).
\end{align*}
Note that this is a presentation of $D+E$.
\item
Let $f \colon X' \longrightarrow X$ be a morphism over $k$ where $X'$ is a quasi-projective scheme.
We say $f^{*}\D$ is well-defined if $f(\Ass(X'))\cap D=\emptyset$ and $s_{0},\dots ,s_{n}$, $t_{0},\dots, t_{m}$ do not vanish at any points of $f(\Ass(X'))$.
In this case, the scheme theoretic inverse image $f^{-1}(D)$ is an effective Cartier divisor on $X'$, which is denoted by $f^{*}D$, and we define 
\[
f^{*} \D = (s_{f^{*}D} ; f^{*}\L, f^{*}s_{0},\dots, f^{*}s_{n}; f^{*}\M, f^{*}t_{0},\dots, f^{*}t_{m} ; f^{*}\psi).
\]
This is a presentation of $f^{*}D$.
\end{enumerate}
\end{defn}

\begin{defn}
Let $X$ be a quasi-projective scheme over $k$.
Let $Y, W \subset X$ be closed subschemes such that $Y\cap \Ass(X)=W\cap \Ass(X)=\emptyset$.
Let $\I_{Y}, \I_{W}$ be the ideals of $Y, W$.
The closed subschemes defined by $\I_{Y}+\I_{W}$ and $\I_{Y}\I_{W}$ are denoted by
$Y\cap W$ and $Y+W$ respectively.
Note that $(Y\cap W)\cap \Ass(X)=\emptyset$ and $(Y+W)\cap \Ass(X)=\emptyset$.

Let 
\begin{align*}
&\Y = (Y; \D_{1}, \dots, \D_{r});\\
&\W = (W; \E_{1},\dots , \E_{s})
\end{align*}
be presentations of $Y, W$.

\begin{enumerate}
\item
We define 
\[
\Y \cap \W = (Y\cap W; \D_{1},\dots ,\D_{r}, \E_{1},\dots, \E_{s}).
\]
This is a presentation of $Y\cap W$.
\item
We define 
\[
\Y+\W = (Y+W; \{\D_{i}+\E_{j}\}_{1\leq i \leq r, 1\leq j \leq s}).
\]
This is a presentation of $Y+W$.
\item
Let $f \colon X' \longrightarrow X$ be a morphism over $k$ where $X'$ is a quasi-projective scheme.
Suppose $f^{*}\D_{i}$ are well-defined for all $i=1,\dots ,r$. (In this case, we say $f^{*}\Y$ is well-defined.)
Then $f^{-1}(Y)\cap \Ass(X')=\emptyset$.
We define 
\[
f^{*}\Y = (f^{-1}(Y); f^{*}\D_{1},\dots, f^{*}\D_{r}).
\]
This is a presentation of $f^{-1}(Y)$.
\end{enumerate}

\end{defn}

\begin{rmk}
Let $f \colon X' \longrightarrow X$ be a morphism over $k$ between quasi-projective schemes over $k$.
Let $Y\subset X$ be a closed subscheme such that $Y\cap (\Ass(X)\cup f(\Ass(X')))=\emptyset$.
Then by \cref{lem:existpres}, there is a presentation $\Y$ of $Y$ such that $f^{*}\Y$ is well-defined.
\end{rmk}

\begin{lem}\label{lem:prefldext}
Let $k \subset k'$ be a field extension.
Let $X$ be a quasi-projective scheme over $k$ and $Y\subset X$ be a closed subscheme such that $Y\cap \Ass(X)=\emptyset$.
Let $\Y = (Y; \D_{1},\dots, \D_{r})$ be a presentation of $Y$ where
\[
\D_{i}=(s_{D_{i}}; \L^{(i)}, s_{0}^{(i)}, \dots s_{n_{i}}^{(i)}; \M^{(i)},t_{0}^{(i)},\dots,t_{m_{i}}^{(i)}).
\]
Then the base change $Y_{k'} \subset X_{k'}$ also satisfies $Y_{k'}\cap \Ass(X_{k'})=\emptyset$ and
\begin{align*}
&\Y_{k'}:=(Y_{k'}; (\D_{1})_{k'},\dots (\D_{r})_{k'})
\end{align*}
is a presentation of $Y_{k'}$ where
\begin{align*}
(\D_{i})_{k'}=(s_{(D_{i})_{k'}}; \L_{k'}^{(i)}, (s_{0}^{(i)})_{k'}, \dots (s_{n_{i}}^{(i)})_{k'}; \M^{(i)}_{k'},(t_{0}^{(i)})_{k'},\dots,(t_{m_{i}}^{(i)})_{k'}).
\end{align*}
\end{lem}
\begin{proof}
Since $X_{k'} \longrightarrow X$ is flat, $\Ass(X')$ is mapped to $\Ass(X)$.
The rest is obvious.
\end{proof}

\begin{cons}[Meromorphic functions associated to presentations]\label{mfunc}\

Let $X$ be an algebraic scheme and $\L$ be an invertible $\O_{X}$-module. 
For a regular section $s \in H^{0}(X, \L)$, we define a section $s^{-1} \in H^{0}(U, \L^{-1})$
where $U=X\setminus \di(s)$ as follows.
Since 
\[
s \colon \O_{U} \longrightarrow \L|_{U}
\]
is an isomorphism, the homomorphism
\[
s^{\vee} \colon \L^{-1}:=\sHom_{\O_{X}}(\L, \O_{X}) \longrightarrow \sHom_{\O_{X}}(\O_{X},\O_{X})
\] 
is also an isomorphism on $U$.
We define $s^{-1}=(s^{\vee})^{-1}(\id_{\O_{U}}) \in H^{0}(U, \L^{-1})$.
Note that $U$ is dense open in $X$ since $s$ is a regular section.

Let $X$ be a quasi-projective scheme over $k$ and $D$ be an effective Cartier divisor on $X$.
Take a presentation of $D$:
\[
\D = (s_{D}; \L, s_{0}, \dots, s_{n};\M, t_{0},\dots, t_{m}; \psi).
\]
Then we have the following isomorphism:
\[
\Psi \colon \L \otimes \M^{-1} \otimes \O_{X}(D)^{-1} \xrightarrow{\psi \otimes \id} \O_{X}(D) \otimes \O_{X}(D)^{-1} \simeq \O_{X}
\]
where the last isomorphism is the canonical isomorphism.
We write
\[
\frac{s_{i}}{s_{D}t_{j}} := \Psi(s_{i}\otimes t_{j}^{-1} \otimes s_{D}^{-1}) \in H^{0}(U, \O_{X})
\]
where $U=X \setminus (D \cup \di(t_{j}))$ is a dense open of $X$.
Let $V = \Spec A \subset U$ be an open subset on which $\L$ and $\M$ are trivial.
Fix isomorphisms $\L|_{V} \simeq \O_{V}$ and $\M|_{V} \simeq \O_{V}$.
By these isomorphisms and $\psi$, we get isomorphisms $\O_{X}(D)|_{V} \simeq \L|_{V} \otimes \M|_{V}^{-1} \simeq \O_{V} $.
Let $f, g, f_{D} \in A$ be elements corresponding to $s_{i}|_{V}, t_{j}|_{V}, s_{D}|_{V}$
via these isomorphisms.
Then $t_{j}^{-1}|_{V}$ and $s_{D}^{-1}|_{V}$ correspond to $g^{-1}$ and $f_{D}^{-1}$ and 
\[
\left. \frac{s_{i}}{s_{D}t_{j}}\right|_{V} \quad \text{corresponds to} \quad \frac{f}{f_{D}g}.
\]

\end{cons}

\subsection{Bounded subsets}
In this section, we fix a field $K$ and a proper set of absolute values $M_{K}$ on $K$.
We fix an algebraic closure $K \subset \overline{K}$.
Let $M=M( \overline{K}) = \{v \mid \text{$v$ is an absolute value on $ \overline{K}$ such that $v|_{K} \in M_{K}$}\}$.

For absolute value $v$ on any field, let $ \epsilon(v)=1$ if $v$ is archimedean and $ \epsilon(v)=0$ if $v$
is non-archimedean.

\begin{defn}
Let $K \subset L \subset  \overline{K}$ be an intermediate field such that $[L:K]<\infty$. 
Let $X$ be an algebraic scheme over $L$.
\begin{enumerate}
\item
Let $U \subset X$ be an open affine subset.
Let $v \in M( \overline{K})$.
A subset $B \subset U( \overline{K})$ is said to be bounded with respect to $U$ and $v$ if
\begin{align*}
\sup_{x\in B}\{|f(x)|_{v}\} < \infty \quad \text{for all $f\in \O(U)$}.
\end{align*}

\item
Let $U \subset X$ be an open affine subset.
An $M_{K}$-bounded family of subsets of $U$ is a family $B=(B_{v})_{v\in M( \overline{K})}$ such that
\begin{enumerate}
\item 
for each $v \in M( \overline{K})$, $B_{v} \subset U( \overline{K})$ is a bounded subset with respect to $U$ and $v$;
\item
for any $f \in \O(U)$ and for any $v_{0}\in M_{K}$, 
\begin{align*}
C_{v_{0},B}(f):=\sup_{\tiny\txt{$v\in M( \overline{K})$ \\$v|v_{0}$}} \sup_{x\in B_{v}}\{  |f(x)|_{v}\} <\infty
\end{align*}
and $C_{v_{0},B}(f)\leq 1$ for all but finitely many $v_{0}\in M_{K}$.
\end{enumerate}
\end{enumerate}
\end{defn}

\begin{rmk}
Let $V \subset U \subset X$ are open affine subsets of $X$.
Let $v \in M$.
Then a bounded subset with respect to $V$ and $v$ is a bounded subset with respect to $U$ and $v$.
Also, an $M_{K}$-bounded family of subsets of $V$ is an $M_{K}$-bounded family of subsets of $U$.
\end{rmk}

\begin{rmk}
Let $U \subset X$ be an open subset.
Finite unions of $M_{K}$-bounded family of subsets of $U$ are also $M_{K}$-bounded family of subsets of $U$.
\end{rmk}

\begin{lem}\label{lem:subod}
Let $K \subset L \subset  \overline{K}$ be an intermediate field such that $[L:K]<\infty$. 
Let $X$ be an algebraic scheme over $L$.
Let $U \subset X$ be an affine open subset and $B=(B_{v})_{v\in M( \overline{K})}$ be an $M_{K}$-bounded family of subsets of $U$.

Let $\{U_{i}\}_{i=1}^{r}$ be an open cover of $U$ where $U_{i}$ are principal open subsets of $U$.
Then for each $i$, there is an  $M_{K}$-bounded family of subsets $B^{i} = (B^{i}_{v})_{v\in M( \overline{K})}$ of $U_{i}$
such that
\[
\bigcup_{i=1}^{r}B_{v}^{i} = B_{v} \quad \text{for every $v\in M( \overline{K})$.}
\]
\end{lem}
\begin{proof}
Let $U=\Spec A$ and $U_{i}=\Spec A_{f_{i}}$ where $f_{i}\in A$.
Since $\bigcup_{i=1}^{r}U_{i}=U$, there are $a_{i},\dots, a_{r}\in A$ such that
\[
a_{1}f_{1}+\cdots + a_{r}f_{r}=1.
\]
Take any $v \in M( \overline{K})$ and $x\in B_{v}$.
Let $v_{0}=v|_{K} \in M_{K}$.
Then 
\begin{align*}
1=  |a_{1}f_{1}+\cdots + a_{r}f_{r}|_{v} \leq r^{ \epsilon(v_{0})} \max_{1\leq i \leq r}\{ C_{v_{0}, B}(a_{i}) \} \max_{1\leq i \leq r}\{  |f_{i}(x)|_{v}\}.
\end{align*}
Thus we have
\[
\frac{1}{\max_{1\leq i \leq r}\{  |f_{i}(x)|_{v}\}} \leq r^{ \epsilon(v_{0})} \max_{1\leq i \leq r}\{ C_{v_{0}, B}(a_{i}) \}.
\]

Now, let $B_{v}^{i} = \{ x \in B_{v} \mid |f_{i}(x)|_{v} = \max_{1\leq j \leq r}\{  |f_{j}(x)|_{v}\}\}$.
Obviously, we have $\bigcup_{i=1}^{r}B_{v}^{i}=B_{v}$.
For every $x$, one of the $f_{i}(x)$ is not zero, and hence $B_{v}^{i} \subset U_{i}( \overline{K})$.
Let $\varphi \in \O(U_{i})$.
Then we can write $\varphi= \alpha/f_{i}^{n}$ for some $ \alpha\in A$ and $n\geq 0$.
Take any $v\in M( \overline{K})$ and let $v_{0} = v|_{K} \in M_{K}$.
Then for $x\in B_{v}^{i}$ we have
\begin{align*}
 |\varphi(x)|_{v} = \frac{| \alpha(x)|_{v}}{|f_{i}(x)|_{v}^{n}}  &=  \frac{ |\alpha(x)|_{v}}{\max_{1\leq j \leq r}\{  |f_{j}(x)|_{v}\}^{n}} \\[3mm]
 &\leq r^{ \epsilon(v_{0})n}\max_{1\leq i \leq r}\{ C_{v_{0}, B}(a_{i}) \}^{n}C_{v_{0}, B}( \alpha).
\end{align*}
Thus $B_{v}^{i}$ is bounded with respect to $U_{i}$ and $v$.
Moreover
\[
\sup_{v|v_{0}}\sup_{x\in B_{v}^{i}}\{|\varphi(x)|_{v}\}\leq  r^{ \epsilon(v_{0})n}\max_{1\leq i \leq r}\{ C_{v_{0}, B}(a_{i}) \}^{n}C_{v_{0}, B}( \alpha)
\]
and the right hand side is $\leq 1$ for all but finitely many $v_{0}$.
This proves $(B_{v}^{i})_{v\in M( \overline{K})}$ is an $M_{K}$-bounded family of subsets of $U_{i}$.

\end{proof}

\begin{prop}\label{prop:bddcov}
Let $K \subset L \subset  \overline{K}$ be an intermediate field such that $[L:K]<\infty$. 
Let $X$ be a projective scheme over $L$.
Let $\{U_{i}\}_{i=1}^{r}$ be an open affine cover of $X$.
Then there are $M_{K}$-bounded family of subsets $(B^{i}_{v})_{v\in M( \overline{K})}$ of $U_{i}$ such that
\[
\bigcup_{i=1}^{r}B_{v}^{i}=X( \overline{K}) \quad \text{for every $v \in M( \overline{K})$.}
\]
\end{prop}
\begin{proof}
By \cref{lem:subod}, it is enough to show the proposition for one specific open affine cover $\{U_{i}\}_{i=1}^{r}$ of $X$.

Fix an embedding $ \iota \colon X \longrightarrow \P^{N}_{L}=\Proj L[x_{0},\dots,x_{N}]$.
Set $U_{i} = \iota^{-1}(D_{+}(x_{i}))$ and $\varphi_{ij}= \iota^{*}(x_{j}/x_{i}) \in \O(U_{i})$.
Let 
\[
B_{v}^{i} = \{ x\in U_{i}( \overline{K}) \mid |\varphi_{ij}(x)|_{v}\leq 1 \text{ for all $j=0,\dots,N$}\}.
\]
Since $\varphi_{ij}, j=0,\dots, N$ generate $\O(U_{i})$ as an $L$-algebra, 
$B_{v}^{i}$ is bounded with respect to $U_{i}$ and $v$.
Take any $f \in \O(U_{i})$.
Then  there is a polynomial $P$ in $N+1$-variables with coefficient in $L$
such that $f=P(\varphi_{i0},\dots , \varphi_{iN})$.
Then for any $v_{0}\in M_{K}$, we have
\begin{align*}
\sup_{\tiny\txt{$v|v_{0}$\\ $v\in M$}} \sup_{x\in B^{i}_{v}} \{|f(x)|_{v}\} &\leq \sup_{v|v_{0}} \{ T^{ \epsilon(v_{0})} \max_{\tiny\txt{$a$ coeff. of $P$}}\{|a|_{v} \} \}\\
&\leq T^{ \epsilon(v_{0})} \sup_{\tiny \txt{$w \in M(L)$ \\ $w|v_{0}$}} \{\max\{ |a|_{w} \mid \text{$a$ coeff. of $P$}  \}\}
\end{align*}
where $T$ is the number of terms in $P$.
The right hand side is $\leq 1$ for all but finitely many $v_{0}$.
Therefore $(B_{v}^{i})_{v\in M}$ is an $M_{K}$-bounded family of subsets of $U_{i}$.

Finally we prove $\bigcup_{i=0}^{N}B_{v}^{i}=X( \overline{K})$.
Fix $v\in M$.
Let $x \in X( \overline{K})$.
Write $ \iota(x) = (a_{0}:\cdots : a_{N})$.
For $i $ such that $|a_{i}|_{v}=\max_{0\leq j \leq N}\{|a_{j}|_{v}\}$,
we have 
\[
|\varphi_{ij}(x)|_{v} = \frac{|a_{j}|_{v}}{|a_{i}|_{v}}\leq 1 \quad \text{for $j=0,\dots, N$.}
\]
Thus $x \in B_{v}^{i}$ and this proves $\bigcup_{i=0}^{N}B_{v}^{i}=X( \overline{K})$.

\end{proof}

\subsection{Local height associated to presentations of closed subschemes}

In this section we fix an infinite field $K$ equipped with a proper set of absolute values $M_{K}$.
We fix an algebraic closure $ \overline{K}$ of $K$ and let 
\[
M = M( \overline{K}) = \left\{  |\ |_{v} \middle| \txt{$  |\ |_{v}$ is an absolute value on $ \overline{K}$ such that \\the restriction $  |\ |_{v}|_{K}$ on $K$ is an element of $M_{K}$}\right\}.
\]

Let $K \subset F \subset L \subset \overline{K}$ be intermediate fields such that $[L:K]<\infty$.
(The field $F$ is going to be the field over which our ambient scheme $X$ is defined and $L$ is going to be the one over which 
the closed subschemes, to which we will associate height functions, are defined.)

\begin{rmk}
For any $F$-scheme $X$, there is a canonical injection $X(L) \longrightarrow X( \overline{K})$ induced by the inclusion $L \subset \overline{K}$.
By this manner, we identify $X(L)$ as a subset of $X( \overline{K})$.
\end{rmk}

{\bf We fix a quasi-projective scheme $X$ over $F$ in the rest of this section.}

\begin{defn}[\bf Local height associated with presentations]\label{def:lochtpre}
Let $Y \subset X_{L}$ be a closed subscheme such that $Y \cap \Ass(X_{L}) = \emptyset$.
Let $\Y = (Y; \D_{1},\dots, \D_{r})$ be a presentation of $Y$ where
\[
\D_{i}=(s_{D_{i}}; \L^{(i)}, s_{0}^{(i)}, \dots s_{n_{i}}^{(i)}; \M^{(i)},t_{0}^{(i)},\dots,t_{m_{i}}^{(i)} ; \psi^{(i)}).
\]
We define the map, which is called local height function associated with presentation $\Y$,
\begin{align*}
\lambda_{\Y} \colon (X_{L}\setminus Y)( \KK) \times M(\KK) \longrightarrow \R
\end{align*}
as follows.
For $(x,v) \in (X_{L}\setminus Y)( \KK) \times M(\KK)$, we set
\begin{align*}
\lambda_{\Y}(x,v):= \min\{ \lambda_{\D_{1}}(x,v), \dots, \lambda_{\D_{r}}(x,v)\}
\end{align*}
where 
\begin{align*}
\lambda_{\D_{i}}(x,v) := \log\left( \max_{0\leq j \leq n_{i}}\left\{ \min_{0\leq k \leq m_{i}}\left\{ \left|\frac{s^{(i)}_{j}}{s_{D_{i}}t^{(i)}_{k}}(x)\right|_{v}  \right\} \right\}\right).
\end{align*}
Here $s^{(i)}_{j}/s_{D_{i}}t^{(i)}_{k}$ is an element of $\O_{X}(X\setminus(D_{i}\cup \di(t_{k}^{(i)})))$
defined in \cref{mfunc}.
If $x \notin (X\setminus(D_{i}\cup \di(t_{k}^{(i)})))( \KK)$, we set $|s^{(i)}_{j}/s_{D_{i}}t^{(i)}_{k} (x)|_{v}=\infty$.
We use the conventions $a < \infty $ for all $a\in \R$ and $\log (\infty)=\infty$.
\end{defn}

\begin{rmk}
Sometime it is convenient to allow $x \in Y( \KK)$.
We define $ \lambda_{\Y}(x,v)=\infty $ if $x \in Y(\KK)$.
\end{rmk}

\begin{lem}\label{lem:sumofdiv}
Let $D, E$ be two effective Cartier divisors on $X_{L}$ and let $\D, \E$ be presentations of $D, E$.
Then 
\begin{align*}
\lambda_{\D+\E} = \lambda_{\D}+ \lambda_{\E} \colon (X_{L}\setminus (D\cup E)) \times M(\KK) \longrightarrow \R.
\end{align*}
\end{lem}
\begin{proof}
Let
\begin{align*}
&\D = (s_{D}; \L, s_{0}, \dots, s_{n};\M, t_{0},\dots, t_{m}; \psi), \\
&\E = (s_{E}; \L', s'_{0}, \dots, s'_{n'};\M', t'_{0},\dots, t'_{m'}; \psi').
\end{align*}
Recall the sum $\D + \E$ is defined as
\begin{align*}
&\D+\E = \\
& \Bigl( s_{D+E};  \L \otimes \L', \{s_{i} \otimes s'_{j}\}_{\tiny \txt{$0\leq i \leq n$\\  $0\leq j \leq n'$}} 
; \M\otimes \M', \{t_{k}\otimes t'_{l}\}_{\tiny \txt{$0\leq k \leq m$\\$0\leq l \leq m'$}} ; \psi \otimes \psi' \Bigr).
\end{align*}
Note that
\begin{align*}
\frac{s_{i} \otimes s'_{j}}{ s_{D+E} t_{k} \otimes t'_{l}} = \frac{s_{i}}{s_{D} t_{k}} \frac{s'_{j}}{s_{E} t'_{l}}
\end{align*}
as elements of $\O_{X_{L}}(X_{L} \setminus (D\cup E \cup \di (t_{k}) \cup \di (t'_{l})))$.
Then for $(x,v) \in (X_{L}\setminus (D\cup E)) \times M(\KK)$,
\begin{align*}
\lambda_{\D+\E}(x,v) & = \log \left( \max_{i,j} \min_{k,l} \left\{  \left|\frac{s_{i}}{s_{D} t_{k}}(x) \frac{s'_{j}}{s_{E} t'_{l}}(x)\right|_{v}  \right\} \right) \\[3mm]
& = \log \left( \max_{i} \min_{k} \left\{  \left|\frac{s_{i}}{s_{D} t_{k}}(x) \right|_{v} \right\} 
\max_{j} \min_{l} \left\{ \left|\frac{s'_{j}}{s_{E} t'_{l}}(x)\right|_{v} \right\} \right) \\[3mm]
&= \lambda_{\D}(x,v) + \lambda_{\E}(x,v).
\end{align*}

\end{proof}

\begin{lem}\label{lem:pullofdiv}
Let $D$ be an effective Cartier divisor on $X_{L}$ and let $\D$ be a presentation of $D$.
Let $X'$ be a quasi-projective scheme over $L$ and $f \colon X' \longrightarrow X_{L}$ a morphism over $L$.
Suppose $f^{*}\D$ is well-defined.
Then 
\[
\lambda_{\D}\circ (f\times \id) = \lambda_{f^{*}\D} \colon (X'\setminus f^{-1}(D))(\KK) \times M(\KK) \longrightarrow \R.
\]
\end{lem}
\begin{proof}
Let
\begin{align*}
&\D = (s_{D}; \L, s_{0}, \dots, s_{n};\M, t_{0},\dots, t_{m}; \psi).
\end{align*}

For $(x,v) \in  (X'\setminus f^{-1}(D))(\KK) \times M(\KK)$, we have
\begin{align*}
\lambda_{f^{*}\D}(x,v) &= \log \max_{i} \min_{j} \left\{ \left| \frac{f^{*}s_{i}}{s_{f^{*}D} f^{*}t_{j}}(x)  \right|_{v}   \right\}\\[3mm]
& =  \log \max_{i} \min_{j} \left\{ \left| \frac{s_{i}}{s_{D} t_{j}}(f(x))  \right|_{v}   \right\} \\[3mm]
& = \lambda_{\D}(f(x),v).
\end{align*}
\end{proof}
 
 \begin{prop}[Basic properties of local height functions associated with presentations]\label{prop:basicprop}
 Let $Y, W \subset X_{L}$ be closed subschemes such that 
 $Y\cap \Ass(X)=W\cap \Ass(X) =  \emptyset$.
 Let $\Y, \W$ be presentations of $Y, W$.
 Then the following hold:
 \begin{enumerate}
 \item
 $ \lambda_{\Y\cap\W} = \min\{ \lambda_{\Y}, \lambda_{\W}\} \colon (X_{L}\setminus (Y\cap W))(\KK) \times M(\KK) \longrightarrow \R$;
 \item
 $ \lambda_{\Y+\W} = \lambda_{\Y} + \lambda_{\W} \colon (X_{L}\setminus (Y\cup W))(\KK) \times M(\KK) \longrightarrow \R$;
 \item
 Let $L \subset L' \subset \KK$ be an intermediate field such that $[L':L]<\infty$.
 Then 
\[
\xymatrix@R=6pt@C=1pt{
\lambda_{\Y} = \lambda_{\Y_{L'}} \ar@{}[r]|(.40) \displaystyle{\colon} & (X_{L}\setminus Y)(\KK) \times M(\KK) \ar@{=}[d] \ar[rrrrr] &&&&& \R.\\
 & (X_{L'}\setminus Y_{L'})(\KK) \times M(\KK) &&&&&
}
\]
\item
Let $X'$ be a quasi-projective scheme over $L$ and $f \colon X' \longrightarrow X$ a morphism over $L$.
Suppose $f^{*}\Y$ is well-defined.
Then 
\[
\lambda_{\Y}\circ (f\times \id) = \lambda_{f^{*}\Y} \colon (X'\setminus f^{-1}(Y))(\KK) \times M(\KK) \longrightarrow \R.
\]
 \end{enumerate}
 \end{prop}
 \begin{proof}
 Let $\Y=(Y; \D_{1},\dots,\D_{r})$ and $\W =(W; \E_{1},\dots,\E_{s})$.
 
 (1) Since $\Y \cap \W=(Y\cap W ; \D_{1},\dots,\D_{r}, \E_{1},\dots,\E_{s})$, we have
 \[
 \lambda_{\Y\cap \W}=\min_{i,j}\{ \lambda_{D_{i}}, \lambda_{\E_{j}}\} = \min\{ \lambda_{\Y}, \lambda_{\W}\}.
 \]
 
 (2) Since $\Y+\W=(Y+W; \{\D_{i}+\E_{j}\}_{i,j})$, by \cref{lem:sumofdiv} we have
 \[
 \lambda_{\Y+\W} = \min_{i,j}\{ \lambda_{\D_{i}+\E_{j}}\} = \min_{i,j}\{ \lambda_{\D_{i}}+ \lambda_{\E_{j}}\} = \lambda_{\Y}+ \lambda_{\W}.
 \]
 
 (3)
 This is trivial from the definition.
 
 (4)
 Since $f^{*}\Y=(f^{-1}(Y); f^{*}\D_{1},\dots ,f^{*}\D_{r})$, by \cref{lem:pullofdiv} we have
 \begin{align*}
  \lambda_{f^{*}\Y} & = \min\{ \lambda_{f^{*}\D_{1}},\dots, \lambda_{f^{*}\D_{r}}\}\\ 
& = \min\{ \lambda_{\D_{1}}\circ (f\times \id),\dots, \lambda_{\D_{1}}\circ (f\times \id)\}
 = \lambda_{\Y} \circ (f \times \id).
 \end{align*}

 \end{proof}

 \begin{lem}\label{lem:restoratpt}
 Let $Y \subset X_{L}$ be a closed subscheme such that 
 $Y\cap \Ass(X) =  \emptyset$.
 Let $\Y$ be a presentation of $Y$.
 Let $L \subset L' \subset \KK$ be an intermediate field.
 Let $x \in X(L')$ be any point.
 For $v, w \in M(\KK)$, if $v|_{L'}=w|_{L'}$, then
 \[
 \lambda_{\Y}(x, v) = \lambda_{\Y}(x, w).
 \] 
 \end{lem}
 \begin{proof}
 In the notation of \cref{def:lochtpre},
 \[
 \frac{s^{(i)}_{j}}{s_{D_{i}}t^{(i)}_{k}}(x)
 \]
 is contained in $L'$. Thus 
 \[
  \left|\frac{s^{(i)}_{j}}{s_{D_{i}}t^{(i)}_{k}}(x)\right|_{v} =  \left|\frac{s^{(i)}_{j}}{s_{D_{i}}t^{(i)}_{k}}(x)\right|_{w}
 \]
 and we get $ \lambda_{\Y}(x, v) = \lambda_{\Y}(x, w)$.
 \end{proof}

 \begin{defn}[$M_{K}$-constant]
 A map $ \gamma \colon M(\KK) \longrightarrow \R_{\geq0}$ is called $M_{K}$-constant if
 $ \gamma$ factors through $M_{K}$, i.e.
 \[
 \xymatrix{
 M(\KK) \ar[d]_{\text{restriction}} \ar[r]^{ \gamma} & \R_{\geq0}\\
 M_{K} \ar[ru]_{ \exists \overline{ \gamma}} &
 }
 \]
 and $ \overline{ \gamma}(v)=0$ for all but finitely many $v\in M_{K}$.
 We denote $ \overline{ \gamma}$ also by $\gamma$ for simplicity.
 
 For any intermediate field $K \subset L \subset \KK$ with $[L:K]<\infty$,
 we also write $ \gamma(v)= \overline{ \gamma}(v|_{K})$ for $v \in M(L)$.
 Note that $ \gamma(v)=0$ for all but finitely many $v \in M(L)$. 
 \end{defn}

 We have concrete expressions of local heights when they are restricted on 
 $M_{K}$-bounded families of subsets.
 
 \begin{lem}\label{lem:divhtbddset}
 Let $D \subset X_{L}$ be an effective Cartier divisor and $\D$ be a presentation of $D$.
 Let $U \subset X_{L}$ be a non-empty open affine subset such that
 \begin{align*}
 &D|_{U} = \di (f) \quad \text{ for a non-zero divisor $f\in \O(U)$};\\
 & \L|_{U} \simeq \O_{U} \simeq \M|_{U}.
 \end{align*}
 Let $B=(B_{v})_{v\in M(\KK)}$ be an $M_{K}$-bounded family of subsets of $U$.
 Then there exists an $M_{K}$-constant $ \gamma$ such that
 \begin{align*}
 \log \frac{1}{|f(x)|_{v}} - \gamma(v) \leq \lambda_{\D}(x,v) \leq \log \frac{1}{|f(x)|_{v}} + \gamma(v)
 \end{align*}
 for all $v \in M(\KK)$ and $x \in B_{v}\setminus D(\KK)$.
 \end{lem}
 \begin{proof}
 Let $\D=(s_{D}; \L, s_{0},\dots, s_{n}; \M, t_{0}.\dots, t_{m}; \psi)$.
 Fix isomorphisms
 \begin{align*}
 & \alpha \colon \L|_{U} \xrightarrow{\sim} \O_{U};\\
 & \beta \colon \M|_{U} \xrightarrow{\sim} \O_{U}
 \end{align*}
 and set 
 \begin{align*}
 &\alpha(s_{i}|_{U}) = g_{i};\\
 & \beta(t_{j}|_{U}) = h_{j}.
 \end{align*}
 Note that $g_{i}, h_{j}$ are non-zero divisors of $\O(U)$.
 We also have the isomorphism
 \begin{align*}
 \alpha\otimes (\beta^{\vee})^{-1} \circ \psi^{-1}  \colon \O_{X}(D)|_{U} \xrightarrow{\sim} \L\otimes \M^{-1}|_{U} \xrightarrow{\sim} \O_{U}.
 \end{align*}
 Then $ \alpha\otimes \beta^{\vee} \circ \psi^{-1} (s_{D}|_{U}) = uf$ for some $u \in \O(U)^{\times}$.
 Then on $V=U\setminus (D\cup \di (t_{j}))$, we have
 \[
 \xymatrix@R=8pt@C=15pt{
 \L \otimes \M^{-1} \otimes \O_{X}(D)^{-1}|_{V}  \ar[r]^{\psi \otimes \id}  & \O_{X}(D) \otimes \O_{X}(D)^{-1}|_{V} \ar[r]^(.75){\sim} & \O_{V}\\
s_{i}|_{V} \otimes t_{j}^{-1}|_{V} \otimes s_{D}^{-1}|_{V}  \ar@{|->}[rr] &     & \displaystyle{\left.\frac{s_{i}}{s_{D}t_{j}}\right|_{V}=\frac{g_{i}}{ufh_{j}}}.
 }
 \] 
 
 Since $s_{0},\dots, s_{n}$ generate $\L$, there are $a_{0},\dots, a_{n} \in \O(U)$ such that
 \begin{align}\label{eq:gen1}
  a_{0}g_{0}+\cdots + a_{n}g_{n}=1.
 \end{align}
 Therefore for any $v\in M(\KK)$ and $x \in B_{v}$, we have
 \begin{align*}
 1 &=  |  a_{0}(x)g_{0}(x)+\cdots + a_{n}(x)g_{n}(x)|_{v} \\[3mm]
 &\leq (n+1)^{ \epsilon(v)} \max_{0\leq i \leq n}\{ |a_{i}(x)|_{v} \}  \max_{0\leq i \leq n}\{|g_{i}(x)|_{v}\}\\[3mm]
 &\leq (n+1)^{ \epsilon(v)} \max_{0\leq i \leq n}\{ \sup_{v'|v_{0}}\sup_{y\in B_{v}}\{|a_{i}(y)|_{v'}\} \}  \max_{0\leq i \leq n}\{|g_{i}(x)|_{v}\} 
 \quad \text{where $v_{0}=v|_{K}$}\\[3mm]
 &=  (n+1)^{ \epsilon(v_{0})} \max_{0\leq i \leq n}\{C_{v_{0},B}(a_{i})\}\max_{0\leq i \leq n}\{|g_{i}(x)|_{v}\}.
 \end{align*}
 By (\ref{eq:gen1}), $\max_{0\leq i \leq n}\{C_{v_{0},B}(a_{i})\} \neq 0$ and we get
 \begin{align}\label{ineq:g}
 \max_{0\leq i \leq n}\{|g_{i}(x)|_{v}\} \geq \frac{1}{ (n+1)^{ \epsilon(v_{0})} \max_{0\leq i \leq n}\{C_{v_{0},B}(a_{i})\}}.
 \end{align}
 Similarly, we take $b_{j} \in \O(U)$ so that $b_{0}h_{0}+\cdots +b_{m}h_{m}=1$ and get
 \begin{align}\label{ineq:h}
 \max_{0\leq j \leq m}\{|h_{j}(x)|_{v}\} \geq \frac{1}{ (m+1)^{ \epsilon(v_{0})} \max_{0\leq j \leq m}\{C_{v_{0},B}(b_{j})\}}.
 \end{align}
 
 For any $v \in M(\KK)$ and $x \in B_{v}\setminus D(\KK)$, let us set $v_{0}= v|_{K}$, then by (\ref{ineq:g}) and (\ref{ineq:h}),
\begin{align*}
& \max_{0\leq i \leq n} \min_{0\leq j \leq m} \left| \frac{s_{i}}{s_{D}t_{j}} (x)\right|_{v}
 = \max_{0\leq i \leq n} \min_{0\leq j \leq m} \left| \frac{g_{i}(x)}{u(x)f(x)h_{j}(x)} \right|_{v}\\[3mm]
& 
\begin{cases}
{\displaystyle \leq C_{v_{0}, B}(u^{-1})  (m+1)^{ \epsilon(v_{0})} \max_{0\leq j \leq m}\{C_{v_{0},B}(b_{j})\} 
\max_{0\leq i \leq n}\{ C_{v_{0},B}(g_{i})\}\frac{1}{|f(x)|_{v}}}\\
\\
{\displaystyle \geq  \frac{1}{\displaystyle C_{v_{0},B}(u) \max_{0\leq j\leq m}\{C_{v_{0}, B}(h_{j})\}
  (n+1)^{ \epsilon(v_{0})} \max_{0\leq i \leq n}\{C_{v_{0},B}(a_{i})\}  }      \frac{1}{|f(x)|_{v}}}.
\end{cases} 
 \end{align*}
 Since the big constant coefficients are $\leq 1$ and $\geq 1$ for all but finitely many $v_{0} \in M_{K}$ respectively,
 we are done.

 \end{proof}

 On projective schemes, local height functions associated with presentations  essentially
 depend only on the closed subschemes.
 
 \begin{prop}[Independence on presentations]\label{prop:indpre}
 Suppose $X$ is projective over $F$.
 Let $Y\subset X_{L}$ be a closed subscheme such that $Y\cap \Ass(X) = \emptyset$.
 Let $\Y$ and $\Y'$ be two presentations of $Y$.
 Then there exists an $M_{K}$-constant $ \gamma$ such that 
 \begin{align*}
 \lambda_{\Y}(x,v) - \gamma(v) \leq \lambda_{\Y'}(x,v) \leq \lambda_{\Y}(x,v) + \gamma(v)
 \end{align*}
 for all $(x,v) \in (X_{L}\setminus Y)(\KK) \times M(\KK)$. 
 \end{prop}
 \begin{proof}
 Let $\Y=(Y; \D_{1},\dots, \D_{r})$ and $\Y' = (Y; \E_{1},\dots, \E_{s})$.
 Then $\Y'' = (Y; \D_{1},\dots, \D_{r}, \E_{1},\dots, \E_{s})$ is also a presentation of $Y$.
 By replacing $\Y'$ with $\Y''$, we may assume $\Y' = (Y; \D_{1},\dots, \D_{r}, \E_{1},\dots, \E_{s})$.
 By arguing inductively on $s$, we may assume $\Y'=(Y; \D_{1},\dots, \D_{r}, \E)$.
 
 By definition, we have $ \lambda_{\Y} \geq \lambda_{\Y'}$.
 To end the proof, it is enough to show that $ \lambda_{\Y'}\geq \lambda_{\Y}- \gamma$
for some $M_{K}$-constant $ \gamma$. 

\begin{claim}
Let $U \subset X_{L}$ be an open affine subset such that all invertible $\O_{X}$-modules in
$\D_{1},\dots, \D_{r}, \E$ are isomorphic to $\O_{U}$.
Let $B=(B_{v})_{v\in M(\KK)}$ be an $M_{K}$-bounded family of subsets of $U$.
Then there exists an $M_{K}$-constant $ \gamma_{U, B}$ such that
\begin{align*}
\lambda_{\Y'}(x,v) \geq \lambda_{\Y}(x,v) - \gamma_{U, B}(v) 
\end{align*}
for all $v \in M(\KK)$ and $x \in B_{v}\setminus Y(\KK)$.
\end{claim}
\begin{claimproof}
Let $D_{1},\dots , D_{r}, E$ be the effective Cartier divisors presented by $\D_{1},\allowbreak\dots , \D_{r}, \E$.
Let $f_{1},\dots ,f_{r}, g \in \O(U)$ be non-zero divisors such that
\begin{align*}
D_{1}|_{U} = \di (f_{1}), \dots, D_{r}|_{U} = \di (f_{r}), E|_{U} = \di (g).
\end{align*}
Since $\D_{1}, \dots \D_{r}$ already form  a presentation of $Y$, we can write
\begin{align*}
g = \sum_{i=1}^{r}a_{i}f_{i}
\end{align*}
for some $a_{i} \in \O(U)$.
Then for any $v \in M(\KK)$ and $x \in B_{v}$, set $v_{0}=v|_{K}$ and get
\begin{align*}
|g(x)|_{v} \leq r^{ \epsilon(v_{0})} \max_{1\leq i \leq r}\{ C_{v_{0}, B}(a_{i}) \} \max_{1\leq i \leq r}\{  |f_{i}(x)|_{v}\}.
\end{align*}
Therefore, if $x \in B_{v}\setminus Y(\KK)$, we have
\begin{align*}
 \lambda_{\Y'}(x,v) & = \min\{ \lambda_{\D_{1}}(x,v), \dots, \lambda_{\D_{r}}(x,v) , \lambda_{\E}(x,v) \}\\[3mm]
 & \geq \min\left\{ \log \frac{1}{|f_{1}(x)|_{v}}, \dots, \log \frac{1}{|f_{r}(x)|_{v}}, \log \frac{1}{|g(x)|_{v}}  \right\} - \gamma(v) \\[3mm]
&\qquad \qquad \qquad \qquad \qquad \qquad \qquad \qquad {\txt{ for some $M_{K}$-constant $ \gamma$ \\( \cref{lem:divhtbddset})}}\\[3mm]
&\geq \min\left\{ \log \frac{1}{|f_{1}(x)|_{v}}, \dots, \log \frac{1}{|f_{r}(x)|_{v}}, \log \frac{1}{\max_{1\leq i \leq r}\{  |f_{i}(x)|_{v}\}}  \right\} \\[3mm]
&\quad -\log \Bigl(r^{ \epsilon(v_{0})} \max_{1\leq i \leq r}\{ C_{v_{0}, B}(a_{i}) \} \Bigr)- \gamma(v)\\[3mm]
&= \min_{1\leq i \leq r} \left\{ \log \frac{1}{|f_{i}(x)|_{v}} \right\} -\log \Bigl(r^{ \epsilon(v_{0})} \max_{1\leq i \leq r}\{ C_{v_{0}, B}(a_{i}) \} \Bigr)- \gamma(v)\\[3mm]
&\geq  \lambda_{\Y}(x,v) - \gamma'(v)-\log\Bigl(r^{ \epsilon(v_{0})} \max_{1\leq i \leq r}\{ C_{v_{0}, B}(a_{i}) \} \Bigr)- \gamma(v) \\[3mm]
&\qquad \qquad \qquad \qquad \qquad \qquad \qquad \qquad { \txt{ for some $M_{K}$-constant $ \gamma'$ \\( \cref{lem:divhtbddset})}}
\end{align*}
and we are done.

\end{claimproof}

Now, take a finite open affine cover $\{U_{i}\}_{i=1}^{n}$ of $X$ consisting of open affines as in the claim.
By \cref{prop:bddcov}, there are $M_{K}$-bounded family of subsets $B^{i} = (B^{i}_{v})_{v\in M(\KK)}$ of $U_{i}$
such that $\bigcup_{i=1}^{n}B^{i}_{v}=X(\KK)$ for each $v \in M(\KK)$.
Then for any $v \in M(\KK)$ and $x \in (X_{L}\setminus Y)(\KK)$, we have
\begin{align*}
\lambda_{\Y'}(x,v) \geq \lambda_{\Y}(x,v) - \max_{1\leq i \leq n}\{ \gamma_{U_{i}, B^{i}}(v)\}.
\end{align*}

 \end{proof}

\subsection{Local height associated to closed subschemes}
In this section we fix an infinite field $K$ equipped with a proper set of absolute values $M_{K}$.
We fix an algebraic closure $ \overline{K}$ of $K$ and let 
\[
M = M( \overline{K}) = \left\{  |\ |_{v} \middle| \txt{$  |\ |_{v}$ is an absolute value on $ \overline{K}$ such that \\the restriction $  |\ |_{v}|_{K}$ on $K$ is an element of $M_{K}$}\right\}.
\]

We will associate local height functions (up to $M_{K}$-bounded function) to closed subschemes of projective schemes over $\KK$.

\begin{defn}
Let $U$ be a quasi-projective scheme over $\KK$.
Two maps
\[
\lambda_{1}, \lambda_{2} \colon U(\KK) \times M(\KK) \longrightarrow \R
\]
are said to be equal up to $M_{K}$-bounded function if there is an $M_{K}$-constant $ \gamma$ such that
 \begin{align*}
 \lambda_{1}(x,v) - \gamma(v) \leq \lambda_{2}(x,v) \leq \lambda_{1}(x,v) + \gamma(v)
 \end{align*}
 for all $(x,v) \in U(\KK) \times M(\KK)$. 
 Note that this defines an equivalence relation.
 Any function in the equivalence class of the constant map $0$ is called $M_{K}$-bounded function.
 When $ \lambda_{1}$ and $ \lambda_{2}$ are equal up to $M_{K}$-bounded function, we write $ \lambda_{1} = \lambda_{2} + O_{M_{K}}(1)$.

 When there is an $M_{K}$-constant $ \gamma$ such that
 \[
 \lambda_{1}(x,v) \leq \lambda_{2}(x,v)+ \gamma(v)
 \]
 for all $(x,v) \in U(\KK) \times M(\KK)$, we write $ \lambda_{1} \leq \lambda_{2} + O_{M_{K}}(1)$. 
\end{defn}

\begin{rmk}
When the functions are allowed to take value $\infty$, we use the same notation under the convention:
\begin{itemize}
\item $a \leq \infty $ for all $a \in \R$;
\item $\infty \pm a = \infty$ for all $a\in \R$;
\item $\infty + \infty = \infty$
\item $ \infty -\infty =0$. 
\end{itemize}
In particular, 
\begin{enumerate}
\item if $ \lambda_{1} \leq \lambda_{2} + O_{M_{K}}(1)$, then
$ \lambda_{1}(x,v) = \infty $ implies $ \lambda_{2}(x,v) = \infty$.

\item if $ \lambda_{1} = \lambda_{2} + O_{M_{K}}(1)$, then
$ \lambda_{1}(x,v) = \infty $ if and only if $ \lambda_{2}(x,v) = \infty$.
\end{enumerate}

\end{rmk}

\begin{cons}\label{cons:lochtsub}
Let $X$ be a projective scheme over $\KK$.
Let $Y\subset X$ be a closed subscheme such that $Y\cap \Ass(X)=  \emptyset$.
Take an intermediate field $K \subset L \subset \KK$ such that 
\begin{itemize}
\item $[L : K]<\infty$;
\item there is a projective scheme $X_{L}$ over $L$ and a closed subscheme $Y_{L} \subset X_{L}$;
\item there is an isomorphism $(X_{L})_{ \KK} \simeq X$ of $\KK$-schemes by which $(Y_{L})_{\KK}$ is isomorphic to $Y$:
\[
\xymatrix{
(X_{L})_{\KK} \ar[r]^{\sim}  & X \\
(Y_{L})_{\KK} \ar@{^{(}->}[u] \ar[r]^{\sim} & Y \ar@{^{(}->}[u].
}
\]
\end{itemize}
 We identify $(X_{L})_{ \KK}, (Y_{L})_{\KK}$ with $X, Y$ via this isomorphisms.
 Note that since $X \longrightarrow X_{L}$ is flat, we have $Y_{L}\cap \Ass(X_{L})$.
 Thus we can take a presentation $\Y$ of $Y_{L}$ and get a map
 \[
 \lambda_{\Y} \colon (X\setminus Y)(\KK) \times M(\KK) \longrightarrow \R.
 \]
 
 Let $L', X_{L'}, Y_{L'}$ be another such a field and schemes and take a presentation $\Y'$ of $Y_{L'}$.
 Then there is a subfield $L'' \subset \KK$ that contains $L, L'$, is finite over $K$, and
 the isomorphism $(X_{L})_{\KK} \xrightarrow{\sim} X \xrightarrow{\sim} (X_{L'})_{\KK}$ is defined over $L''$. 
 Then we get isomorphisms of $L''$-schemes
 \[
 \xymatrix{
(X_{L})_{L''} \ar[r]^{\sim}  & (X_{L'})_{L''} \\
(Y_{L})_{L''} \ar@{^{(}->}[u] \ar[r]^{\sim} & (Y_{L'})_{L''}  \ar@{^{(}->}[u].
}
 \]
 and identify them by these isomorphisms.
 Then $\Y_{L''}$ and $\Y'_{L''}$ are presentations of the same closed subscheme and therefore
 by \cref{prop:basicprop}(3) and \cref{prop:indpre}, there exists an $M_{K}$-constant $ \gamma$ such that
 \begin{align*}
 \lambda_{\Y}(x,v) - \gamma(v) \leq \lambda_{\Y'}(x,v) \leq \lambda_{\Y}(x,v) + \gamma(v)
 \end{align*}
 for all $(x,v) \in (X\setminus Y)(\KK) \times M(\KK)$. 
 
 Therefore $ \lambda_{\Y}$ and $ \lambda_{\Y'}$ are equal up to $M_{K}$-bounded function.
\end{cons}

\begin{defn}[\bf Local height associated with closed subschemes]
In the notation of \cref{cons:lochtsub}, 
any map $ (X\setminus Y)(\KK) \times M(\KK) \longrightarrow \R$ which is equal to $ \lambda_{\Y}$ up to $M_{K}$-bounded function
 is called a height function associated with $Y$ and denoted by $ \lambda_{Y}$.
 This is well-defined by \cref{cons:lochtsub}.
\end{defn}

\begin{rmk}
Note that $ \lambda_{Y}$ is determined up to $M_{K}$-bounded function by $Y$.
\end{rmk}

\begin{prop}[Good choice of local heights]\label{prop:goodchoice}
Let $K \subset F \subset \KK$ be an intermediate field such that $[F:K]<\infty$.
Let $X$ be a projective scheme over $F$.
Let $Y\subset X$ be a closed subscheme such that $Y\cap \Ass(X)=  \emptyset$.
Then we can take a local height function $ \lambda_{Y_{\KK}}$ associated with $Y_{\KK}$ so that
the following holds.
For any intermediate field $F \subset L \subset \KK$ with $[L:F]<\infty$,
$ \lambda_{Y_{\KK}}$ induces a map $(X_{L}\setminus Y_{L})(L) \times M(L) \longrightarrow \R$:
\[
\xymatrix{
(X_{L}\setminus Y_{L})(L) \times M(\KK) \ar@{^{(}->}[r] \ar[d]_{\id \times (\ )|_{L}} & 
(X_{\KK}\setminus Y_{\KK})(\KK) \times M(\KK) \ar[r]^(.80){ \lambda_{Y_{\KK}}} & \R \\
(X_{L}\setminus Y_{L})(L) \times M(L) \ar@/_{10pt}/[rru]_{\exists} &&.
}
\]
Namely, $ \lambda_{Y_{\KK}}= \lambda_{\Y}$ works for any presentation $\Y$ of $Y$.
\end{prop}
\begin{proof}
This follows from \cref{lem:restoratpt}.
\end{proof}

Let us prove basic properties of local height functions.
Let us first recall operations of closed subschemes.
Let $X$ be an algebraic scheme and $Y, W \subset X$ be closed subschemes
and $\I_{Y}, \I_{W}$ be ideals of them.
We define:
\begin{enumerate}
\item
$Y\cap W \subset X$  is the closed subscheme defined by $\I_{Y}+\I_{W}$;
\item
$Y+W \subset X$ is the closed subscheme defined by $\I_{Y}\I_{W}$;
\item
$Y\cup W \subset X$  is the closed subscheme defined by $\I_{Y}\cap\I_{W}$.
\end{enumerate}
Note that
\[
\I_{Y}\I_{W} \subset \I_{Y}\cap \I_{W}
\]
and thus $Y\cup W \subset Y+W \subset X$ as schemes.
Also, $\Supp(Y\cup W)=\Supp(Y+W)$.

\begin{thm}[Basic properties of local heights]\label{thm:basiclocht}
Let $X$ be a projective scheme over $\KK$.
Let $Y, W \subset X$ be closed subschemes such that $Y\cap \Ass(X) = W\cap \Ass(X) =  \emptyset$.
Fix local heights $\lambda_{Y}$ and $ \lambda_{W}$ associated with $Y$ and $W$.
\begin{enumerate}
\item\label{propertyintersec}
Fix local heights $ \lambda_{Y\cap W}$ associated with $Y\cap W$.
Then 
\[
\lambda_{Y\cap W} = \min\{ \lambda_{Y}, \lambda_{W}\} + O_{M_{K}}(1)
\]
on $(X\setminus (Y\cap W))(\KK)\times M(\KK)$.
\item
Fix local heights $ \lambda_{Y+W}$ associated with $Y+W$.
Then 
\[
\lambda_{Y+W} = \lambda_{Y}+\lambda_{W} + O_{M_{K}}(1)
\]
on $(X\setminus (Y\cup W))(\KK)\times M(\KK)$.

\item\label{propertycontain}
If $Y\subset W$ as schemes, then 
\[
\lambda_{Y} \leq \lambda_{W} + O_{M_{K}}(1)
\]
on $(X\setminus Y)(\KK)\times M(\KK)$.

\item
Fix local heights $ \lambda_{Y\cup W}$ associated with $Y\cup W$.
Then
\begin{align*}
&\max\{ \lambda_{Y}, \lambda_{W}\} \leq \lambda_{Y\cup W} +O_{M_{K}}(1) ;\\
&\lambda_{Y\cup W}  \leq \lambda_{Y} + \lambda_{W} + O_{M_{K}}(1)
\end{align*}
on $(X\setminus (Y\cup W))(\KK)\times M(\KK)$.
\item
If $\Supp Y \subset \Supp W$, then there exists a constant $C>0$ such that 
\[
\lambda_{Y} \leq C \lambda_{W} +O_{M_{K}}(1)
\]
on $(X\setminus Y)(\KK)\times M(\KK)$.
\item\label{propertypullback}
Let $ \varphi \colon X' \longrightarrow X$ be a morphism where $X'$ is a projective scheme over $\KK$.
Suppose $\varphi(\Ass(X'))\cap Y=  \emptyset$.
Then we can define $ \lambda_{\varphi^{-1}(Y)}$ ,where $\varphi^{-1}(Y)$ is the scheme theoretic preimage,
and 
\[
\lambda_{Y} \circ (\varphi \times \id)= \lambda_{\varphi^{-1}(Y)} + O_{M_{K}}(1)
\]
on $(X'\setminus \varphi^{-1}(Y))(\KK) \times M(\KK)$.
\end{enumerate}
\end{thm}
\begin{proof}
This follows form \cref{prop:basicprop,cons:lochtsub}.
\end{proof}

\section{Arithmetic distance function}
In this section we fix an infinite field $K$ equipped with a proper set of absolute values $M_{K}$.
We fix an algebraic closure $ \overline{K}$ of $K$ and let 
\[
M = M( \overline{K}) = \left\{  |\ |_{v} \middle| \txt{$  |\ |_{v}$ is an absolute value on $ \overline{K}$ such that \\the restriction $  |\ |_{v}|_{K}$ on $K$ is an element of $M_{K}$}\right\}.
\]

\begin{defn}[Arithmetic distance function]
Let $X$ be a projective scheme over $\KK$.
Let $ \Delta_{X} \subset X \times X$ be the diagonal.
Suppose $ \Delta_{X}\cap \Ass(X\times X)=  \emptyset$.
The local height function  $ \lambda_{ \Delta_{X}}$ associated with $ \Delta_{X}$ is called arithmetic distance function on $X$
and denoted by $ \delta_{X}$:
\[
\delta_{X} = \lambda_{ \Delta_{X}} \colon (X\times X \setminus \Delta_{X})(\KK) \times M(\KK) \longrightarrow \R.
\]
Note that this is determined up to $M_{K}$-bounded function.
We set $ \delta_{X}(x,x,v)= \infty$ for $x \in X(\KK)$ and $v \in M(\KK)$.
\end{defn}

\begin{rmk}
If $X$ is reduced and all the irreducible components have $\dim \geq 1$, then 
$ \Delta_{X}\cap \Ass(X\times X)=  \emptyset$.
\end{rmk}

\begin{prop}\label{prop:goodchoicedist}
Let $K \subset F \subset \KK$ be an intermediate field such that $[F:K]<\infty$.
Let $X$ be a projective scheme over $F$.
Suppose $ \Delta_{X}\cap \Ass(X \times X)=  \emptyset$.
Then we can take an arithmetic distance function $ \delta_{X_{\KK}} $ so that
the following holds.
For any intermediate field $F \subset L \subset \KK$ with $[L:F]<\infty$,
$ \delta_{X_{\KK}}$ induces the indicated map:
\[
\xymatrix{
(X\times X \setminus \Delta_{X})(L) \times M(\KK) \ar@{^{(}->}[r] \ar[d]_{\id \times (\ )|_{L}} & 
(X\times X \setminus \Delta_{X})(\KK) \times M(\KK) \ar[r]^(.80){ \delta_{X_{\KK}}} & \R \\
(X\times X \setminus \Delta_{X})(L) \times M(L) \ar@/_{10pt}/[rru]_{\exists} &&.
}
\]
Namely, $ \delta_{X_{\KK}} = \lambda_{ \widetilde{\Delta_{X}}}$ works for any presentation $\widetilde{\Delta_{X}}$ of $ \Delta_{X}$.

\end{prop}
\begin{proof}
This follows from \cref{prop:goodchoice}.
\end{proof}

\begin{prop}[Basic properties of arithmetic distance function]\label{prop:arithdistbasic}
Let $X$ be a projective scheme over $\KK$ and suppose $ \Delta_{X} \cap (X\times X) = \emptyset$.
Fix an arithmetic distance function $ \delta_{X}$.

\begin{enumerate}
\item {\rm (Symmetry)} There exists an $M_{K}$-constant $ \gamma$ such that
\begin{align*}
| \delta_{X}(x,y,v) - \delta_{X}(y,x,v)| \leq \gamma(v)
\end{align*}
for all $(x,y,v) \in (X\times X \setminus \Delta_{X})(\KK) \times M(\KK)$. 

\item {\rm (Triangle inequality I)} There exists an $M_{K}$-constant $ \gamma$ such that
\begin{align*}
\delta_{X}(x,z,v) + \gamma(v) \geq  \min\{  \delta_{X}(x,y,v), \delta_{X}(y,z,v)  \}
\end{align*}
for all $(x,y,z,v) \in X(\KK)^{3} \times M(\KK)$.

\item {\rm (Triangle inequality II)} Let $Y \subset X$ be a closed subscheme such that $Y\cap \Ass(X) =  \emptyset$.
Fix a local height $ \lambda_{Y}$ associated with $Y$.
Then there is an $M_{K}$-constant $ \gamma$ such that
\begin{align*}
\lambda_{Y}(y,v) + \gamma(v) \geq \min\{ \lambda_{Y}(x,v), \delta_{X}(x,y, v) \}
\end{align*}
for all $(x,y,v) \in X(\KK)^{2} \times M(\KK)$.

\item Let $y \in X(\KK)$. Suppose $y \in X$, as a closed point, is not an associated point of $X$.
Fix a local height function $ \lambda_{y}$ associated with $\{y\}$.
Then there is an $M_{K}$-constant $ \gamma$ such that
\begin{align*}
| \delta_{X}(x,y, v)- \lambda_{y}(x,v)| \leq \gamma(v)
\end{align*}
for all $(x,v)\in (X(\KK)\setminus \{y\})\times M(\KK)$.
 
\end{enumerate}

\end{prop}
\begin{proof}
(1) Apply \cref{thm:basiclocht}(\ref{propertypullback}) to the automorphism
\[
\sigma \colon X\times X \longrightarrow X\times X; (x,y) \mapsto (y,x).
\]
Note that $ \sigma^{-1}( \Delta_{X}) = \Delta_{X}$.

(2) Consider the following projections:
\[
\xymatrix{
& X\times X \times X \ar[dl]_{\pr_{12}} \ar[d]_{\pr_{13}} \ar[rd]^{\pr_{23}} &\\
X\times X & X\times X & X\times X.
}
\]
Note that $\pr_{ij}$ are flat and hence we have $\pr_{ij}(\Ass(X\times X \times X)) \subset \Ass(X\times X)$.
Then up to $M_{K}$-bounded functions, we have
\begin{align*}
&\min\{ \delta_{X}(x,y,v), \delta_{X}(y,z,v)\} \\[3mm]
&= \min\{ \lambda_{ \Delta_{X}}(x,y,v), \lambda_{\Delta_{X}}(y,z,v)\}\\[3mm]
&=\min\{ \lambda_{\pr_{12}^{-1}( \Delta_{X})}(x,y,z,v)  , \lambda_{\pr_{23}^{-1}( \Delta_{X})}(x,y,z,v)  \}  
	  \qquad \text{by \cref{thm:basiclocht} (\ref{propertypullback})}  \\[3mm]
&= \lambda_{\pr_{12}^{-1}( \Delta_{X})\cap \pr_{23}^{-1}( \Delta_{X})}(x,y,z,v)
	   \qquad \text{by \cref{thm:basiclocht} (\ref{propertyintersec})}  \\[3mm]
&\leq \lambda_{\pr_{13}^{-1}( \Delta_{X})}(x,y,z,v) 
\qquad  \txt{ $\pr_{12}^{-1}( \Delta_{X})\cap \pr_{23}^{-1}( \Delta_{X}) \subset \pr_{13}^{-1}( \Delta_{X})$ \\
	and \cref{thm:basiclocht}(\ref{propertycontain})} \\[3mm]
& = \lambda_{ \Delta_{X}}(x,z,v) = \delta_{X}(x,z,v).
\end{align*}

(3)
Consider the following diagram:
\[
\xymatrix{
X\times X \ar[d]_{\pr_{1}} \\
\qquad X \supset Y.
}
\]
Since $\pr_{1}$ is flat, we have $\pr_{1}(\Ass(X\times X)) \subset \Ass(X)$.
Then up to $M_{K}$-bounded functions, we have
\begin{align*}
&\min\{ \lambda_{Y}(x,v) , \delta_{X}(x,y,v)\}\\[3mm]
& = \min\{ \lambda_{\pr_{1}^{-1}(Y)}(x,y,v), \lambda_{ \Delta_{X}}(x,y,v)\}\\[3mm]
&= \lambda_{\pr_{1}^{-1}(Y)\cap \Delta_{X}}(x,y,v) 
	&&\text{by \cref{thm:basiclocht} (\ref{propertyintersec})} \\[3mm]
&\leq \lambda_{\pr_{2}^{-1}(Y)} (x,y,v) 
	&&\txt{ $\pr_{1}^{-1}(Y) \cap \Delta_{X} \subset \pr_{2}^{-1}(Y)$\\
	 and  \cref{thm:basiclocht}(\ref{propertycontain})}\\[3mm]
& = \lambda_{Y}(y,v).
\end{align*}

(4)
Consider the embedding 
\[
i \colon X \longrightarrow X \times X; x \mapsto (x,y).
\]
Since $y \notin \Ass(X)$, $i(\Ass(X))\cap \Delta_{X} =  \emptyset$.
Thus by  \cref{thm:basiclocht} (\ref{propertypullback}), up to $M_{K}$-bounded function, we have
\[
\delta_{X}(x,y,v) = \lambda_{ \Delta_{X}}(x,y,v) = \lambda_{i^{-1}( \Delta_{X})}(x,v) = \lambda_{y}(x,v).
\]

\end{proof}

\section{Local heights on quasi-projective schemes}

On quasi-projective schemes, a closed subscheme does not necessarily determine a local height function up to $M_{K}$-bounded
function.
However, it is still possible to attach a local height function up to local height function of ``boundary''. 

\subsection{Good projectivization}
In this subsection we work over a field $k$.

\begin{defn}
Let $X$ be a quasi-projective scheme over $k$.
A good projectivization of $X$ is a projective scheme $\Xbar$ over $k$ and an open immersion
\[
i \colon X \longrightarrow \Xbar
\]
over $k$ such that $\Ass(\Xbar) \subset X$.
\end{defn}
\begin{rmk}
Good projectivization is compatible with base change.
That is, let $k\subset k'$ be arbitrary field extension.
Let $X$ be a quasi-projective scheme over $k$ and $\Xbar$ be a projective scheme over $k$.
Let $i \colon X \longrightarrow \Xbar$ be a morphism over $k$.
Then $i$ is a good projectivization if and only if the base change $i_{k'} \colon X_{k'} \longrightarrow \Xbar_{k'}$
is a good projectivization.
Indeed, the equivalence of being open immersion follows from fpqc descent.
Let $p \colon \Xbar_{k'} \longrightarrow \Xbar$ be the projection.
Since $p$ is flat, we have $p(\Ass(\Xbar_{k'}))= \Ass(\Xbar)$. 
Thus $\Ass(\Xbar) \subset X$ if and only if $\Ass(\Xbar_{k'}) \subset p^{-1}(X) = X_{k'}$. 
\end{rmk}

Note that the condition $\Ass(\Xbar) \subset X$ implies $X$ is dense in $\Xbar$.
Let $Y \subset X$ be a closed subscheme such that $Y\cap \Ass(X)=  \emptyset$.
For any good projectivization $X \subset \Xbar$ and
any closed subscheme $ \widetilde{Y} \subset \Xbar$ such that $ \widetilde{Y}\cap X= Y$,
we have $ \widetilde{Y} \cap \Ass(\Xbar) =  \emptyset$.

For a given quasi-projective scheme $X$ and an open immersion $i \colon X \longrightarrow \Xbar$
into a projective scheme, there are two ways to modify $\Xbar$ so that it becomes a good projectivization.

\begin{lem}\label{lem:goodpro}
Let $X$ be a quasi-projective scheme over $k$.
Let $i \colon X \longrightarrow \Xbar$ be an open immersion into a projective scheme $\Xbar$
such that $X$ is dense in $\Xbar$.
\begin{enumerate}
\item\label{lem:goodpro:clsub} There is a closed subscheme $j \colon \Xbar_{0} \subset \Xbar$ such that 
	\begin{enumerate}
	\item $j$ is the identity on the underlying topological spaces;
	\item $j$ is isomorphic on $X$;
	\item $\Ass(\Xbar_{0}) \subset X$.
	\end{enumerate}
	In particular, $\Xbar_{0}$ is a good projectivization of $X$.
\item  Let $Z \subset \Xbar$ be a closed subscheme with support $\Xbar \setminus X$.
Let $\pi \colon \Xbar_{1} \longrightarrow \Xbar$ be the blow-up of $\Xbar$ along $Z$.
Then 
	\begin{enumerate}
	\item $\pi^{-1}(X)$ is dense in $\Xbar_{1}$ and isomorphic to $X$;
	\item $\pi^{-1}(Z) \subset \Xbar_{1}$ is an effective Cartier divisor. 
	\end{enumerate}
	In particular, $\Xbar_{1}$ is a good projectivization of $X$.
\end{enumerate}
\end{lem}
\begin{proof}

(1)
First take an open affine subset $U = \Spec A \subset \Xbar$.
Let $I$ be the radical ideal such that $X\cap U = U \setminus V(I)$, where $V(I)$ is the set of primes containing $I$.
Let 
\[
0 = \q_{1}\cap \cdots \cap \q_{r}
\]
be a minimal primary decomposition of $0$.
By changing labels, we may assume $I \not\subset \sqrt{\q_{i}}$ for $i=1,\dots, s$ and $I \subset \sqrt{\q_{j}}$ for $j=s+1,\dots, r$.
Then the ideal
\[
\a = \q_{1}\cap \cdots \cap \q_{s}
\]
is independent of the choice of minimal primary decomposition by the 2nd uniqueness theorem \cite[Theorem 4.10]{am}.
Since $X$ is dense in $\Xbar$, $ \sqrt{\q_{j}}$, $j=s+1,\dots, r$ are not minimal primes.
Thus $\a$ is a nilpotent ideal.
Set $U_{0} = \Spec A/\a$.
By the uniqueness of $\a$, this construction glue together and define desired $\Xbar_{0}$.

(2)
This follows from basic properties of blow-ups.
\end{proof}

\begin{lem}\label{lem:prodgp}
Let $X, X'$ be quasi-projective schemes over $k$.
Let
\[
i \colon X \longrightarrow \Xbar,\quad j \colon X' \longrightarrow \Xbar'
\]
be good projectivizations.
Then the product
\[
i \times j \colon X \times X' \longrightarrow \Xbar \times \Xbar'
\]
is also a good projectivization.
\end{lem}
\begin{proof}
Since $i, j$ are open immersions, so is $i\times j$.
Since the projection $\pr_{i}, i=1,2$ from $\Xbar \times \Xbar'$ to each factor is flat, 
we have
\begin{align*}
\Ass(\Xbar \times \Xbar') & \subset \pr_{1}^{-1}(\Ass(\Xbar)) \cap \pr_{2}^{-1}(\Ass(\Xbar')) \\
& \subset \pr_{1}^{-1}(X) \cap \pr_{2}^{-1}(X') = X \times X'.
\end{align*}
\end{proof}

\subsection{Meromorphic functions and meromorphic sections}

In this subsection, we briefly recall the definition of the sheaf of meromorphic functions and meromorphic sections of invertible sheaves
on Noetherian schemes (cf.\ \cite[7.1.1]{liu} for the sheaf of meromorphic functions).

\subsubsection{The sheaf of meromorphic functions}

\begin{defn}
Let $X$ be a Noetherian scheme.
\begin{enumerate}
\item For any open subset $U \subset X$, define 
\[
\cR_{X}(U) := \left\{ a \in \O_{X}(U) \middle| \txt{the germ $a_{x} \in \O_{X,x}$ at any point\\ $x \in U$ is a non-zero divisor}  \right\}.
\]
This defines a sheaf $\R_{X}$ on $X$.

\item
For any open subset $U \subset X$, define 
\[
\K_{X}'(U) := \cR_{X}(U)^{-1}\O_{X}(U).
\]
This defines a presheaf $\K_{X}'$ of rings on $X$.

\item
The sheafification of  $\K_{X}'$ is denoted by $\K_{X}$ and is called the sheaf of meromorphic functions on $X$.
\end{enumerate}
\end{defn}

\begin{prop}
Let $X$ be a Noetherian scheme.
\begin{enumerate}
\item We have $\O_{X} \subset \K_{X}' \subset \K_{X}$ as presheaves on $X$.
\item For any $x \in X$, we have $\K_{X,x}' \simeq \K_{X,x} \simeq \Frac \O_{X,x}$ as $\O_{X,x}$-algebras.
\item Let $U\subset X$ be an open subset such that $\Ass(X) \subset U$.
Let $i \colon U \longrightarrow X$ be the inclusion.
Then we have
\begin{align*}
& \text{$\O_{X} \longrightarrow i_{*}\O_{U}$  is injective;}\\
& \text{$\K_{X} \longrightarrow i_{*}\K_{U}$  is isomorphic.}
\end{align*}
In particular, we have $\K_{X}(X) \simeq \K_{X}(U)$ by the restriction.

\item Let $U \subset X$ be an open affine subset.
Then we have
\[
\K_{X}(U) = \Frac \O_{X}(U).
\]
\end{enumerate}
\end{prop}
\begin{proof}
See \cite[7.1.1]{liu}.
\end{proof}

\subsubsection{Cartier divisors and meromorphic sections}\label{subsubsec:divandsec}

\begin{defn}
Let $X$ be a Noetherian scheme.
A Cartier divisor on $X$ is a global section of the sheaf $\K_{X}^{\times}/\O_{X}^{\times}$.
Note that a Cartier divisor is represented by a set of pairs $\{(U_{i}, f_{i})\}_{i}$ where
$\{U_{i}\}_{i}$ is an open cover of $X$, $f_{i} \in \K_{X}^{\times}(U_{i})$, and 
$f_{i}/f_{j} \in \O_{X}^{\times}(U_{i}\cap U_{j})$.
\end{defn}

Let $D$ be a Cartier divisor on a Noetherian scheme $X$.
Take a representation $\{(U_{i}, f_{i})\}_{i}$ of $D$.
Define
\[
\I_{D} := \text{the $\O_{X}$-submodule of $\K_{X}$ generated by $f_{i}$'s}
\]
and 
\[
\O_{X}(D):=\sHom_{\O_{X}}(\I_{D}, \O_{X}).
\]
Note that $\I_{D}$ and $\O_{X}(D)$ are invertible $\O_{X}$-modules.
Note that $\O_{X}(D)$ can be embedded into $\K_{X}$ canonically by 
\[
\O_{X}(D)|_{U_{i}} \longrightarrow \K_{X}|_{U_{i}}, \quad  \varphi \mapsto \frac{\varphi(f_{i})}{f_{i}}.
\]
Now, let $U\subset X$ be the largest open subset of $X$ such that $\I_{X}|_{U} \subset \O_{U}$.
Define 
\[
s_{D} \in H^{0}(U, \O_{X}(D)) = \Hom_{\O_{U}}(\I_{D}|_{U}, \O_{U})
\]
to be the inclusion $\I_{D}|_{U} \longrightarrow \O_{U}$.
Note that
\begin{itemize}
\item $\Ass(X) \subset U$;
\item $s_{D}$ is a regular section of $\O_{X}(D)$ on $U$.
\end{itemize}

On the other hand, let $\L$ be an invertible $\O_{X}$-module on $X$ and
$U\subset X$ be an open subset such that $\Ass(X) \subset U$.
Let $s \in H^{0}(U, \L)$ be a regular section.
Then we can construct a Cartier divisor in the following way.
Take an open cover $\{U_{i}\}_{i}$ of $X$ such that 
\[
\varphi_{i} \colon \L|_{U_{i}} \xrightarrow{\sim} \O_{U_{i}}.
\]
Let $f_{i} \in \K_{X}^{\times}(U_{i})$ be the image of $s|_{U\cap U_{i}}$ by
\[
H^{0}(U\cap U_{i}, \L) \xrightarrow{\sim} H^{0}(U\cap U_{i}, \O_{X}) \subset \K_{X}(U\cap U_{i}) = \K_{X}(U_{i}).
\]
Note that since $s$ is a regular section, $f_{i}$ is a unit element.
The isomorphism $\varphi_{j}\circ \varphi_{i}^{-1} \colon \O_{U_{ij}} \longrightarrow \O_{U_{ij}}$, 
where $U_{ij} = U_{i}\cap U_{j}$, is a multiplication by a unit $u_{ij} \in \O_{U_{ij}}^{\times}(U_{ij})$.
Thus we get $f_{j}/f_{i}=u_{ij}$ as elements in $\K_{X}(U_{ij})$.
Thus $\{ (U_{i}, f_{i}) \}_{i}$ defines a Cartier divisor $D$ on $X$.
By the constructions, we have
\begin{align*}
& \L \simeq \O_{X}(D);\\
& H^{0}(U, \L) \simeq H^{0}(U, \O_{X}(D)), \quad s \longleftrightarrow s_{D}|_{U}.
\end{align*}

\subsection{Boundary functions}

In this subsection we fix an infinite field $K$ equipped with a proper set of absolute values $M_{K}$.
We fix an algebraic closure $ \KK$ of $K$.

\begin{defn}
Let $X$ be a quasi-projective scheme over $K$ or $\KK$.
Let $ \lambda_{1}$, $ \lambda_{2}$, and $ \lambda_{3}$ be maps from $X(\KK) \times M(\KK)$ to $\R$.
\begin{enumerate}
\item We write 
\[
\lambda_{1} \ll \lambda_{2}
\]
if there is a positive real number $C>0$ such that
\[
\lambda_{1}(x,v) \leq C\lambda_{2}(x,v) 
\]
for all $(x,v) \in X(\KK) \times M(\KK)$.

\item
We write
\[
\lambda_{1} \gl \lambda_{2}
\]
if $ \lambda_{1} \ll \lambda_{2}$ and $ \lambda_{2} \ll \lambda_{1}$.

\item We write 
\[
\lambda_{1} = \lambda_{2} + O( \lambda_{3})
\]
if 
\[
| \lambda_{1} - \lambda_{2}| \ll  \lambda_{3}.
\]
\end{enumerate}
\end{defn}

\begin{rmk}
We say $ \lambda_{1} \ll \lambda_{2}$ up to $M_{K}$-bounded function or write $ \lambda_{1} \ll \lambda_{2} + O_{M_{K}}(1)$
if there is a positive constant $C>0$ and an $M_{K}$-constant $ \gamma$ such that
\[
\lambda_{1}(x,v) \leq C( \lambda_{2}(x,v) + \gamma(v) )
\]
for all $(x,v) \in X(\KK) \times M(\KK)$.
The same is applied to $ \lambda_{1} \gl \lambda_{2}$.
\end{rmk}

\begin{lem}\label{lem:pullbackbdfcproper}
Let $\varphi \colon X' \longrightarrow X$ be a proper morphism between quasi-projective schemes over $\KK$.
Let 
\[
i \colon X \longrightarrow \Xbar, \quad j \colon X' \longrightarrow \Xbar'
\]
be good projectivizations.
Let $Z \subset \Xbar$ and $Z' \subset \Xbar'$ be closed subschemes such that
\[
\Supp Z = \Xbar \setminus X, \quad \Supp Z' = \Xbar' \setminus X'.
\]
Fix local height functions $ \lambda_{Z}$ and $ \lambda_{Z'}$.
Then we have
\[
\lambda_{Z} \circ ((i\circ \varphi) \times \id) \gl \lambda_{Z'} \circ (j \times \id) 
\]
on $X'(\KK) \times M(\KK)$ up to $M_{K}$-bounded function.
\end{lem}
\begin{proof}
Let $ \Gamma \subset \Xbar' \times \Xbar$ be the graph of $\varphi$, i.e.\ 
the scheme theoretic closure of $(j, i\circ \varphi) \colon X' \longrightarrow  \Xbar' \times \Xbar$.
Let $ \Gamma_{0} \subset \Gamma$ be a closed subscheme which is a good projectivization of $X'$
constructed in  \cref{lem:goodpro} (\ref{lem:goodpro:clsub}).
Then we get the following commutative diagram:
 \[
\xymatrix{
X' \ar[d]_{\varphi} \ar[r]^{j} \ar@/^{20pt}/[rr]^{k} & \Xbar'  & \Gamma_{0} \ar[l]_{p_{1}} \ar[dl]^{p_{2}} \\
X \ar[r]_{i} & \Xbar &
}
\]
where $p_{i}$ are the morphisms induced by projections and $k$ is the open immersion.
Since $\varphi$ is proper, $k$ is the base change of $i$ along $p_{2}$.
Thus
\[
p_{2}^{-1}(Z) = \Gamma_{0} \setminus X' = p_{1}^{-1}(Z')
\]
as sets, where the last equality follows from the definition of the graph.
In particular, $Z \cap p_{2}(\Ass( \Gamma_{0})) =  \emptyset$ and $Z' \cap p_{1}(\Ass( \Gamma_{0}))$.
Therefore, we get
\begin{align*}
\lambda_{Z} \circ ((i\circ \varphi) \times \id) &= \lambda_{Z} \circ ((p_{2}\circ k) \times \id) + O_{M_{K}}(1)\\
& = \lambda_{p_{2}^{-1}(Z)} \circ (k \times \id) + O_{M_{K}}(1)\\
& \gl \lambda_{p_{1}^{-1}(Z')} \circ (k \times \id) + O_{M_{K}}(1)\\
&= \lambda_{Z'} \circ (j\times \id) + O_{M_{K}}(1).
\end{align*}

\end{proof}

\begin{cor}\label{cor:indofbdryfunc}
Let $X$ be a quasi-projective scheme over $\KK$.
Let
\[
i \colon X \longrightarrow \Xbar_{1}, \quad j \colon X \longrightarrow \Xbar_{2}
\]
be two good projectivizations.
Let $Z_{1} \subset \Xbar_{1}$ and $Z_{2} \subset \Xbar_{2}$ be closed subschemes such that
\[
\Supp Z_{1} = \Xbar_{1}\setminus X, \quad \Supp Z_{2} = \Xbar_{2} \setminus X.
\]
Fix local height functions $ \lambda_{Z_{1}}$ and $ \lambda_{Z_{2}}$.
Then we have
\[
\lambda_{Z_{1}} \circ (i \times \id) \gl \lambda_{Z_{2}} \circ (j \times \id) 
\]
on $X(\KK) \times M(\KK)$ up to $M_{K}$-bounded function.
\end{cor}

\begin{proof}
This follows from \cref{lem:pullbackbdfcproper}.
\end{proof}

\begin{defn}
Let $X$ be a quasi-projective scheme over $\KK$.
Let us take a good projectivization $X \subset \Xbar$.
Let $Z \subset \Xbar$ be a closed subscheme with $\Supp Z = \Xbar \setminus X$.
A function $ \lambda$ such that $ \lambda \gl \lambda_{Z}$ up to $M_{K}$-bounded function
for such $Z$ is denoted by $ \lambda_{\partial X}$ and called a boundary function of $X$.
By \cref{cor:indofbdryfunc}, if $ \lambda$ and $ \lambda'$ are two boundary functions of $X$, we have
\[
 \lambda \gl \lambda'
\]
up to $M_{K}$-bounded function.

Let $ \lambda_{1}$ and $ \lambda_{2}$ be two maps from $X(\KK) \times M(\KK)$ to $\R \cup \{\infty\}$.
The property
\[
\lambda_{1}  \leq  \lambda_{2} + O( \lambda_{\partial X}) + O_{M_{K}}(1)
\]
is independent of the choice of $ \lambda_{\partial X}$.
(We use the convention that $\infty$ is the max element in $\R \cup \{\infty\}$, $\infty \pm a = \infty $ for any $a\in \R$, 
$\infty + \infty = \infty$, and $\infty-\infty=0$.)
We write 
\[
\lambda_{1} = \lambda_{2} + O( \lambda_{\partial X})+ O_{M_{K}}(1)
\]
if $\lambda_{1}  \leq  \lambda_{2} + O( \lambda_{\partial X})+ O_{M_{K}}(1)$ 
and $\lambda_{2}  \leq  \lambda_{1} + O( \lambda_{\partial X})+ O_{M_{K}}(1)$.
In other words,
\[
| \lambda_{1}(x,v)- \lambda_{2}(x,v)| \leq C \lambda_{Z}(x,v) + \gamma(v), \quad (x,v) \in X(\KK) \times M(\KK)
\]
for some $C>0$ and some $M_{K}$-constant $ \gamma$.

If $\lambda_{1} = \lambda_{2} + O( \lambda_{\partial X})+ O_{M_{K}}(1)$, we say
$ \lambda_{1}$ and $ \lambda_{2}$ are equal up to boundary function.
\end{defn}

\begin{rmk}
Since any $ \lambda_{\partial X}$ takes finite values on $X(\KK) \times M(\KK)$,
$\lambda_{1}  \leq  \lambda_{2} + O( \lambda_{\partial X})+ O_{M_{K}}(1)$ and $ \lambda_{1}(x,v)= \infty$ imply
$ \lambda_{2}(x,v) = \infty$.
\end{rmk}

\begin{rmk}\label{rmk:bdrpull}
By \cref{lem:pullbackbdfcproper}, boundary functions are compatible with proper morphisms.
That is, let $\varphi \colon X' \longrightarrow X$ be a proper morphism between quasi-projective schemes over $\KK$.
Then for any boundary functions $ \lambda_{\partial X}$ and $ \lambda_{\partial X'}$ on $X$, $X'$ respectively, we have
\[
\lambda_{\partial X} \circ (\varphi \times \id) \gl \lambda_{\partial X'}
\]
on $X'(\KK) \times M(\KK)$ up to $M_{K}$-bounded function.
\end{rmk}

\begin{rmk}
On a projective scheme, a boundary function is nothing but an $M_{K}$-bounded function.
\end{rmk}

\begin{lem}\label{lem:bdrprod}
Let $X, X'$ be quasi-projective schemes over $\KK$.
Then for any boundary functions $ \lambda_{\partial X}$ and $ \lambda_{\partial X'}$ on $X$ and $X'$,
we have
\[
\lambda_{\partial (X\times X')} \gl \lambda_{\partial X}\circ (\pr_{1}\times \id) + \lambda_{\partial X'}\circ (\pr_{2}\times \id)
\]
up to $M_{K}$-bounded function.
Here $\pr_{i}$ is the projection from $X \times X'$ to each factor.
\end{lem}
\begin{proof}
Let
\[
i \colon X \longrightarrow \Xbar,\quad j \colon X' \longrightarrow \Xbar'
\]
be good projectivizations.
Then the product
\[
i \times j \colon X \times X' \longrightarrow \Xbar \times \Xbar'
\]
is a good projectivization by \cref{lem:prodgp}.
Let $Z=(\Xbar \setminus X)_{\rm red}$ and $Z' = (\Xbar' \setminus X')_{\rm red}$. 
Then $(\Xbar \times \Xbar') \setminus (X \times X') = \pr_{1}^{-1}(Z) \cup \pr_{2}^{-1}(Z')$ as sets.
Thus
\[
\lambda_{((\Xbar \times \Xbar') \setminus (X \times X') )_{\rm red}}
 \gl  \lambda_{Z} \circ (\pr_{1} \times \id ) + \lambda_{Z'} \circ (\pr_{2} \times \id)
\]
up to $M_{K}$-bounded function.
\end{proof}

\subsection{Local height functions on quasi-projective schemes}

In this subsection we fix an infinite field $K$ equipped with a proper set of absolute values $M_{K}$.
We fix an algebraic closure $ \KK$ of $K$.

We will attach  local height functions to closed subschemes of quasi-projective schemes up to boundary functions.
On a projective scheme, local heights associated with presentations of a closed subscheme are the same up to
$M_{K}$-bounded function.
On a quasi-projective scheme, it is generalized to the following theorem, namely the same statement is true if we replace
$M_{K}$-bounded function with boundary function.  

\begin{thm}\label{thm:induptobdy}
Let $K \subset F \subset L \subset \KK$ be intermediate fields such that $[L:K]<\infty$.
Let $X$ be a quasi-projective scheme over $F$.
Let $Y \subset X_{L}$ be a closed subscheme such that $Y \cap \Ass(X_{L}) =  \emptyset$.
Let $\Y$ and $\Y'$ be two presentations of $Y$.
Then $ \lambda_{\Y}$ and $ \lambda_{\Y'}$ are equal up to boundary function, i.e.
\[
\lambda_{\Y} = \lambda_{\Y'} + O( \lambda_{\partial X_{\KK}})+ O_{M_{K}}(1)
\]
on $X(\KK)\times M(\KK)$.
\end{thm}

\begin{cor}\label{cor:pullbackbfc}
Let $ \varphi \colon X' \longrightarrow X$ be a morphism between quasi-projective schemes over $\KK$.
Then 
\[
\lambda_{\partial X} \circ (\varphi \times \id) = O( \lambda_{\partial X'})+ O_{M_{K}}(1).
\]
\end{cor}
\begin{proof}
Let $K \subset F \subset \KK$ be an intermediate field such that $[F:K] < \infty $ and 
$X, X'$, and $\varphi$ are defined over $F$.
Let us denote models over $F$ by $X_{F}, X'_{F}$, and $\varphi_{F}$.
(Note that we take $F$ large enough so that $X_{F}$ and $X'_{F}$ are quasi-projective.)
Take a good projectivization $i \colon X_{F} \longrightarrow \Xbar $ over $F$.
Let $Z = \Xbar \setminus X_{F}$ with reduced structure.
Take a presentation $\cZ$ of $Z$.
Then $i^{*}\cZ$ is a presentation of the empty subscheme $  \emptyset \subset X_{F}$.
It is enough to show the statement for $ \lambda_{\partial X} = \lambda_{\cZ}$.
By the functoriality, we have
\begin{align*}
\lambda_{\cZ} \circ (\varphi \times \id) = \lambda_{\varphi^{*}\cZ}.
\end{align*}
Note that $\varphi^{*}\cZ$ is a presentation of $  \emptyset \subset X'_{F}$.
By \cref{thm:induptobdy}, we get $ \lambda_{\varphi^{*}\cZ} = O( \lambda_{\partial X'})$.

\end{proof}

\begin{defn}[Local heights on quasi-projective schemes]
Let $X$ be a quasi-projective scheme over $\KK$.
Let $Y \subset X$ be a closed subscheme such that $Y \cap \Ass(X) =  \emptyset$.
Let $K \subset F \subset \KK$ be an intermediate field such that $[F:K] < \infty $
and there exist a quasi-projective scheme $X_{F}$ over $F$ and a closed subscheme
$Y_{F} \subset X_{F}$ such that
 \[
\xymatrix{
(X_{F})_{\KK} \ar[r]^{\sim} & X\\
(Y_{F})_{\KK} \ar[r]^{\sim} \ar@{^{(}->}[u] & Y \ar@{^{(}->}[u]
}
\] 
commutes.

Let $\Y$  be a presentation of $Y_{F}$.
Any map $X(\KK) \times M(\KK) \longrightarrow \R\cup \{\infty\}$ which is equal to $ \lambda_{\Y}$
up to boundary function is called a height function associated with $Y$ and denoted by $ \lambda_{Y}$.
This function $ \lambda_{Y}$ is determined up to boundary function and the definition is independent of the choice 
of $F, X_{F}$, and the presentation $\Y$.
\end{defn}

To prove  \cref{thm:induptobdy}, we start with the following lemma, by which we can actually
give an alternative  definition of local height on quasi-projective schemes.

\begin{lem}\label{lem:localheightbyprojectivization}
Let $X$ be a quasi-projective scheme over $\KK$.
Let $Y \subset X$ be a closed subscheme such that $Y \cap \Ass(X) =  \emptyset$.
Let
\[
i \colon X \longrightarrow \Xbar_{1}, \quad j \colon X \longrightarrow \Xbar_{2}
\]
be two good projectivizations.
Let $ \widetilde{Y}_{1} \subset \Xbar_{1}$ and $ \widetilde{ Y}_{2} \subset \Xbar_{2}$ be closed subschemes such that
\[
\widetilde{Y}_{1} \cap X = Y, \quad \widetilde{Y}_{2} \cap X = Y.
\] 
Fix local heights $ \lambda_{ \widetilde{Y}_{1}}$ and $ \lambda_{ \widetilde{Y}_{2}}$.
Then we have
\[
 \lambda_{ \widetilde{Y}_{1}}\circ (i \times \id) =  \lambda_{ \widetilde{Y}_{2}} \circ (j \times \id)+ O( \lambda_{\partial X})+ O_{M_{K}}(1)
\]
on $X(\KK) \times M(\KK)$.
\end{lem}
\begin{proof}
Let $ \Gamma \subset \Xbar_{1} \times \Xbar_{2}$ be
the scheme theoretic closure of $(i, j\circ \varphi) \colon X \longrightarrow  \Xbar_{1} \times \Xbar_{2}$.
Let $ \Gamma_{0} \subset \Gamma$ be a closed subscheme which is a good projectivization of $X$
constructed in  \cref{lem:goodpro} (\ref{lem:goodpro:clsub}).
Then we get the following commutative diagram:
 \[
\xymatrix{
X \ar[rd]_{j}  \ar[r]^{i} \ar@/^{20pt}/[rr]^{k} & \Xbar_{1}  & \Gamma_{0} \ar[l]_{p_{1}} \ar[dl]^{p_{2}} \\
 & \Xbar_{2} &
}
\]
where $p_{i}$ are the morphisms induced by projections and $k$ is the open immersion.
Then we have
\[
p_{1}^{-1}( \widetilde{Y}_{1}) \cap X = Y = p_{2}^{-1}( \widetilde{Y}_{2}) \cap X
\]
as closed subschemes.

\begin{claim}
Let $V$ be an algebraic scheme, $U \subset V$ an open subset, and $Z= V \setminus U$ equipped with the reduced structure.
Let $W_{1}, W_{2}$ be closed subschemes of $V$ such that $W_{1}\cap U= W_{2} \cap U$ as closed subschemes.
Then there is a non-negative integer $n$ such that
\[
W_{1} \subset W_{2} + n Z.
\]
\end{claim}
\begin{claimproof}
Let $\I_{W_{1}}, \I_{W_{2}}$, and $\I_{Z}$ be the ideal sheaf of $W_{1}, W_{2}$, and $Z$.
Let
\[
\mathcal{F} = {\rm Im} (\I_{W_{2}} \longrightarrow \O_{V} \longrightarrow \O_{V}/\I_{W_{1}}).
\]
Then the support of this coherent sheaf $ \mathcal{F}$ is contained in $Z$.
Thus there is $n \geq 0$ such that $\I_{Z}^{n} \mathcal{F}=0$.
This implies the map
\[
\I_{Z}^{n}\I_{W_{2}} \longrightarrow \O_{V} \longrightarrow \O_{V}/\I_{W_{1}}
\]
is zero.
This is what we wanted.
\end{claimproof}

Let $Z = \Gamma_{0} \setminus X$ with reduced structure.
By the claim, there is $n \geq 0$ such that
\begin{align*}
p_{1}^{-1}( \widetilde{Y}_{1}) \subset p_{2}^{-1}( \widetilde{Y}_{2}) + n Z.
\end{align*}

Note that for $l=1,2$ we have
\begin{align*}
&\widetilde{Y}_{l} \cap \Ass( \Xbar_{l}) = \widetilde{Y}_{l} \cap \Ass(X) =  \emptyset\\
&\widetilde{Y}_{l} \cap p_{l}(\Ass( \Gamma_{0})) = \widetilde{Y}_{l} \cap p_{l}(\Ass(X)) =  \emptyset
\end{align*}
and therefore $ \lambda_{ \widetilde{Y}_{l}}$ and $ \lambda_{p_{l}^{-1}( \widetilde{Y}_{l})}$ are well-defined.

Then we get
\begin{align*}
\lambda_{ \widetilde{Y}_{1}}\circ (i \times \id) & = \lambda_{ \widetilde{Y}_{1}} \circ (p_{1}\times \id ) \circ (k \times \id) \\
& = \lambda_{p_{1}^{-1}( \widetilde{Y}_{1})} \circ (k \times \id) + O_{M_{K}}(1) \\
& \leq \lambda_{p_{2}^{-1}( \widetilde{Y}_{2})} \circ (k \times \id) + n \lambda_{Z} \circ (k \times \id) + O_{M_{K}}(1) \\
& = \lambda_{ \widetilde{Y}_{2}} \circ (j \times \id) + n \lambda_{Z} \circ (k \times \id) + O_{M_{K}}(1)
\end{align*}
We can get the opposite inequality by switching the role of $ \widetilde{Y}_{1}$ and $ \widetilde{Y}_{2}$ and we are done.

\end{proof}

Now we prove \cref{thm:induptobdy}.

\begin{proof}[Proof of \cref{thm:induptobdy}]\

\begin{claim}\label{claim:divpreext}
Let $X$ be a quasi-projective scheme over $F$ and $D$ be an effective Cartier divisor on $X$.
Let $\D$ be a presentation of $D$.
Then there exists 
\begin{itemize}
\item a good projectivization $X \subset \Xbar$ such that $\Xbar \setminus X$ is the support of an effective Cartier divisor on $\Xbar$;
\item an effective Cartier divisor $ \tD$ on $\Xbar$ such that $\tD |_{X}= D$; 
\item a presentation $ \widetilde{\D}$ of $\tD$ and a presentation $\E$ of an effective Cartier divisor whose support is $\Xbar \setminus X$,
\end{itemize}
such that
\[
\widetilde{\D}|_{X} = \D + \E|_{X}.
\]

\end{claim}
\begin{claimproof}
Let 
\[
\D = (s_{D}; \L, s_{0},\dots, s_{n}; \M, t_{0},\dots, t_{m}; \psi).
\]
Since $s_{i}$'s and $t_{j}$'s are generating sections, they define a morphism from $X$ to projective spaces 
$\P^{n}$ and $\P^{m}$:
\begin{align*}
\Psi_{\L} &\colon X \longrightarrow \P^{n}=\Proj F[x_{0},\dots, x_{n}]=:P_{1}\\
\Psi_{\M} &\colon X \longrightarrow \P^{m}= \Proj F[y_{0},\dots, y_{m}] =: P_{2}.
\end{align*}
Let take an immersion
\[
\alpha \colon X \longrightarrow \P^{N} =: P_{3}
\]
into some projective space $\P^{N}$.
Let $\Xbar' \subset P_{1} \times P_{2} \times P_{3}$ be the scheme theoretic closure of the immersion
$(\Psi_{\L}, \Psi_{M}, \alpha) \colon X \longrightarrow P_{1} \times P_{2} \times P_{3}$.
Let $\Xbar$ be the blow up of $\Xbar'$ along $(\Xbar' \setminus X)_{ {\rm red}}$.
We get the following diagram:
 \[
\xymatrix{
X \ar[rrr]^{(\Psi_{\L}, \Psi_{M}, \alpha)} \ar[rd]_{i} &&& P_{1} \times P_{2} \times P_{3}\\
& \Xbar \ar[r] & \Xbar' \ar[ru] &
}
\]
where $i$ is the open immersion and it is a good projectivization by \cref{lem:goodpro}.
We identify $X$ with its images via this diagram.
Let $E \subset \Xbar$ be the exceptional divisor of the blow up $\Xbar \longrightarrow \Xbar'$.
Note that $E = \Xbar \setminus X$ as sets.

Let 
\[
p_{i} \colon \Xbar \longrightarrow  P_{1} \times P_{2} \times P_{3} \longrightarrow P_{i}
\]
be the morphism induced by the projections.
Let 
\begin{align*}
\widetilde{\L} & = p_{1}^{*} \O_{P_{1}}(1), \quad \widetilde{s_{0}} = p_{1}^{*}x_{0}, \dots , \widetilde{s_{n}} = p_{1}^{*}x_{n} \\
\widetilde{\M} &= p_{2}^{*} \O_{P_{2}}(1), \quad \widetilde{t_{0}} = p_{2}^{*}y_{0}, \dots , \widetilde{t_{m}} = p_{2}^{*}y_{m}.
\end{align*}
By construction, we have
\begin{align*}
\widetilde{\L}|_{X} &= \L, \quad \widetilde{s_{i}}|_{X} = s_{i}\\
\widetilde{\M}|_{X} &= \M, \quad \widetilde{t_{j}}|_{X} = t_{j}.
\end{align*}
Since $\Ass(\Xbar) = \Ass(X)$, $ \widetilde{s_{i}}$ and $ \widetilde{t_{j}}$ are regular sections.

By the isomorphism
 \[
\xymatrix{
\widetilde{\L} \otimes \widetilde{\M}^{-1}|_{X} \ar[r]^{\sim} & \L \otimes \M^{-1} \ar[r]^{\sim}_{\psi} & \O_{X}(D)
}
\]
the regular section $s_{D} \in H^{0}(X, \O_{X}(D))$ can be regraded as a regular section 
\[
s_{D} \in H^{0}(X, \widetilde{\L} \otimes \widetilde{\M}^{-1}).
\]
Since $\Ass(\Xbar) \subset X$, this is a meromorphic section.
Thus by \cref{subsubsec:divandsec}, this defines a Cartier divisor $D'$ on $\Xbar$ and by construction 
\[
D'|_{X} = D.
\]
Recall that $\Xbar \setminus X$ is the support of the exceptional divisor $E$.
Thus there is a non-negative integer $d\geq 0$ such that 
\[
\tD := D'+dE 
\]
is effective.
Note that $\tD|_{X}=D$ and
\begin{align*}
\O_{\Xbar}(\tD) \simeq \O_{\Xbar}(D') \otimes \O_{\Xbar}(dE) \simeq \widetilde{\L} \otimes \widetilde{\M}^{-1} \otimes \O_{\Xbar}(dE).
\end{align*}
We denote this isomorphism by $\chi$ (from right to left).
Take a presentation of $dE$:
\[
\E = (s_{dE}; \L'; s_{0}', \dots, s_{n'}'; \M' , t_{0}', \dots, t_{m'}'; \psi').
\]
Then 
\[
\widetilde{\D} := (s_{\tD}; \widetilde{\L} \otimes \L', \{ \widetilde{s_{i}}\otimes s_{i'}';\};
\widetilde{\M} \otimes \M', \{ \widetilde{t_{j}} \otimes t_{j'}'\} ; \chi \circ \id \otimes \psi')
\]
is a presentation of $\tD$.
Then we get
\[
\widetilde{\D}|_{X} = \D + \E|_{X}
\]
and we are done.
\end{claimproof}

\begin{claim}\label{claim:clsubpreext}
Let $X$ be a quasi-projective scheme over $F$ and $Y\subset X$ be a closed subscheme such that
$Y \cap \Ass(X)=  \emptyset$.
Let 
\[
\Y = (Y; \D_{1}, \dots, \D_{r})
\]
be a presentation of $Y$.
Then there exist
\begin{enumerate}
\item a good projectivization $i \colon X \longrightarrow \Xbar$;
\item a closed subscheme $ \tY \subset \Xbar$ such that $\tY \cap X=Y$ (in particular, $\tY \cap \Ass(\Xbar)=  \emptyset$);
\item a presentation $ \widetilde{\Y}$ of $\tY$;
\item presentations $\E_{1}, \dots, \E_{r}$ of effective Cartier divisors on $\Xbar$ whose supports are contained in 
$\Xbar \setminus X$,
\end{enumerate}
such that
\[
\widetilde{\Y}|_{X} = (Y; \D_{1}+\E_{1}|_{X}, \dots, \D_{r}+\E_{r}|_{X}).
\]
\end{claim}
\begin{claimproof}
Let $D_{i}$ be the divisors presented by $\D_{i}$.
By the definition of presentation, $Y=D_{1}\cap \cdots \cap D_{r}$.

For each $\D_{i}$, apply \cref{claim:divpreext} and get:
\begin{itemize}
\item a good projectivization $j_{i} \colon X \longrightarrow \Xbar_{i}$;
\item an effective Cartier divisor $ \tD_{i} \subset \Xbar_{i}$ and its presentation $ \widetilde{\D_{i}}$;
\item a presentation $\E_{i}$ of an effective Cartier divisor $E_{i} \subset \Xbar_{i}$ such that
$\Supp E_{i} = \Xbar_{i} \setminus X$,
\end{itemize}
satisfying 
\begin{align*}
&j_{i}^{*}\tD_{i} = D_{i}\\
&j_{i}^{*} \widetilde{\D_{i}} = \D_{i} + j_{i}^{*}\E_{i}.
\end{align*}

Let $\Xbar'$ be the scheme theoretic image of the morphism
$(j_{i})_{i=1}^{r} \colon X \longrightarrow \prod_{i=1}^{r} \Xbar_{i}$.
Let $\Xbar \subset \Xbar'$ be the closed subscheme which is a good projectivization of $X$ 
constructed in \cref{lem:goodpro} (\ref{lem:goodpro:clsub}).
We get the following diagram:
 \[
\xymatrix@R=8pt{
X \ar[rdd]_{i} \ar[rrr]^{(j_{i})_{i=1}^{r}} & & & \prod_{i=1}^{r} \Xbar_{i} \ar[r]^(.60){\pr_{i}} & \Xbar_{i}\\
&  &  \Xbar' \ar[ru] & & \\
& \Xbar \ar[ru] \ar@/_10pt/[rruu]_{ \alpha} \ar@/_19pt/[rrruu]_{p_{i}}&&&
}
\]
where $ \alpha$ is the induced closed immersion, $\pr_{i}$ is the $i$-th projection, $p_{i}= \pr_{i} \circ \alpha$,
and $i$ is the good projectivization.
Note that $p_{i}(\Ass(\Xbar)) = p_{i}(\Ass(X)) = j_{i}(\Ass(X)) = \Ass(\Xbar_{i})$.
Therefore we can pull-back presentations on $\Xbar_{i}$ by $p_{i}$.

Let
\[
\tY = p_{1}^{*}\tD_{1} \cap \cdots \cap p_{r}^{*}\tD_{r}.
\]
Then $\tY \cap \Ass(\Xbar)=  \emptyset$ and $\tY \cap X = Y$.
Let 
\[
\widetilde{\Y} = (\tY; p_{1}^{*}\widetilde{\D_{1}}, \dots, p_{r}^{*} \widetilde{\D_{r}}).
\]
This is a presentation of $\tY$ and 
\[
\widetilde{\Y}|_{X} = (Y; \D_{1}+p_{1}^{*}\E_{1}|_{X}, \dots, \D_{r}+ p_{r}^{*}\E_{r}|_{X} ).
\]
Note that $p_{i}^{*}\E_{i}$ is a presentation of $p_{i}^{*}E_{i}$ and $\Supp p_{i}^{*}E_{i} \subset \Xbar \setminus X$.
\end{claimproof}

Now let $X$ be a quasi-projective scheme over $F$ and $Y \subset X_{L}$ be a closed subscheme such that
$Y \cap \Ass(X_{L}) =  \emptyset$.
Let $\Y = (Y; \D_{1}, \dots, \D_{r})$ be a presentation of $Y$.
Take $i \colon X_{L} \to \Xbar, \tY, \widetilde{\Y}$, and $\E_{i}$ as in \cref{claim:clsubpreext} (apply on $X_{L}$ and set $F=L$).
Then on $X(\KK) \times M(\KK)$, we have
\begin{align*}
\lambda_{\Y} = \min\{ \lambda_{\D_{1}}, \dots, \lambda_{\D_{r}} \}
\end{align*}
and 
\begin{align*}
\lambda_{ \widetilde{\Y}} \circ (i\times \id) = \min\{ \lambda_{\D_{1}}+ \lambda_{\E_{1}} \circ (i\times \id),\dots,
 \lambda_{\D_{r}}+ \lambda_{\E_{r}} \circ (i\times \id)\}.
\end{align*}
Thus
\begin{align*}
&\lambda_{ \widetilde{\Y}} \circ (i\times \id) \leq \lambda_{\Y} + \max\{ \lambda_{\E_{1}} \circ (i\times \id),\dots, \lambda_{\E_{r}} \circ (i\times \id)\}\\
&\lambda_{ \widetilde{\Y}} \circ (i\times \id) \geq \lambda_{\Y} +  \min\{ \lambda_{\E_{1}} \circ (i\times \id),\dots, \lambda_{\E_{r}} \circ (i\times \id)\}.
\end{align*}
Since $\E_{i}$ are presentations of divisors supported in $\Xbar \setminus X_{L}$, we have
\begin{align*}
&\max\{ \lambda_{\E_{1}} \circ (i\times \id),\dots, \lambda_{\E_{r}} \circ (i\times \id)\} = O( \lambda_{\partial X})+ O_{M_{K}}(1)\\
& \min\{ \lambda_{\E_{1}} \circ (i\times \id),\dots, \lambda_{\E_{r}} \circ (i\times \id)\}= O( \lambda_{\partial X})+ O_{M_{K}}(1).
\end{align*}
This implies 
\[
\lambda_{\Y} = \lambda_{ \widetilde{\Y}} \circ (i\times \id) + O( \lambda_{\partial X})+ O_{M_{K}}(1).
\]

Note that $ \lambda_{ \widetilde{\Y}}$ is a local height on a projective scheme $\Xbar$.
Hence the statement follows from \cref{lem:localheightbyprojectivization}.

\end{proof}

\begin{thm}[Basic properties of local heights on quasi-projective schemes]\label{thm:basiclochtqp}
Let $X$ be a quasi-projective scheme over $\KK$.
Let $Y, W \subset X$ be closed subschemes such that $Y\cap \Ass(X) = W\cap \Ass(X) =  \emptyset$.
Fix local heights $\lambda_{Y}$ and $ \lambda_{W}$ associated with $Y$ and $W$.
\begin{enumerate}
\item\label{propertyintersecqp}
Fix local heights $ \lambda_{Y\cap W}$ associated with $Y\cap W$.
Then 
\[
\lambda_{Y\cap W} = \min\{ \lambda_{Y}, \lambda_{W}\} + O( \lambda_{\partial X})+ O_{M_{K}}(1)
\]
on $X(\KK)\times M(\KK)$.
\item
Fix local heights $ \lambda_{Y+W}$ associated with $Y+W$.
Then 
\[
\lambda_{Y+W} = \lambda_{Y}+\lambda_{W} + O( \lambda_{\partial X})+ O_{M_{K}}(1)
\]
on $X(\KK)\times M(\KK)$.

\item\label{propertycontainqp}
If $Y\subset W$ as schemes, then 
\[
\lambda_{Y} \leq \lambda_{W} + O( \lambda_{\partial X})+ O_{M_{K}}(1)
\]
on $X(\KK)\times M(\KK)$.

\item
Fix local heights $ \lambda_{Y\cup W}$ associated with $Y\cup W$.
Then
\begin{align*}
&\max\{ \lambda_{Y}, \lambda_{W}\} \leq \lambda_{Y\cup W} +O( \lambda_{\partial X})+ O_{M_{K}}(1) ;\\
&\lambda_{Y\cup W}  \leq \lambda_{Y} + \lambda_{W} + O( \lambda_{\partial X})+ O_{M_{K}}(1)
\end{align*}
on $X(\KK)\times M(\KK)$.
\item
If $\Supp Y \subset \Supp W$, then there exists a constant $C>0$ such that 
\[
\lambda_{Y} \leq C \lambda_{W} +O( \lambda_{\partial X})+ O_{M_{K}}(1)
\]
on $X(\KK)\times M(\KK)$.
\item\label{propertypullbackqp}
Let $ \varphi \colon X' \longrightarrow X$ be a morphism where $X'$ is a quasi-projective scheme over $\KK$.
Suppose $\varphi(\Ass(X'))\cap Y=  \emptyset$.
Then we can define $ \lambda_{\varphi^{-1}(Y)}$ ,where $\varphi^{-1}(Y)$ is the scheme theoretic preimage,
and 
\[
\lambda_{Y} \circ (\varphi \times \id) = \lambda_{\varphi^{-1}(Y)} + O( \lambda_{\partial X'})+ O_{M_{K}}(1)
\]
on $X'(\KK) \times M(\KK)$.

\item 
Let $ \varphi \colon X' \longrightarrow X$ be a proper morphism where $X'$ is a quasi-projective scheme over $\KK$.
Then for any boundary functions $ \lambda_{\partial X}$ and $ \lambda_{\partial X'}$, we have
\[
\lambda_{\partial X} \circ (\varphi \times \id ) \gl \lambda_{\partial X'}
\]
on $X'(\KK)\times M(\KK)$ up to $M_{K}$-bounded function.
\end{enumerate}
\end{thm}
\begin{proof}
Follows from the definition of local height associated with presentations and \cref{prop:basicprop,cor:pullbackbfc,thm:induptobdy,rmk:bdrpull}.
\end{proof}

\subsection{Arithmetic distance function on quasi-projective schemes}

In this subsection we fix an infinite field $K$ equipped with a proper set of absolute values $M_{K}$.
We fix an algebraic closure $ \KK$ of $K$.

\begin{defn}
Let $X$ a quasi-projective scheme over $\KK$.
Suppose $ \Delta_{X} \cap \Ass(X \times X) =  \emptyset$.
Then a local height function associated with $ \Delta_{X}$ is denoted by $ \delta_{X}$ and called an arithmetic distance function on $X$:
\[
\delta_{X}(x,y,v) = \lambda_{ \Delta_{X}}(x,y,v) 
\]
for $(x,y, v) \in (X \times X)(\KK) \times M(\KK)$.
Note that this is determined up to boundary function on $X \times X$.
\end{defn}

\begin{prop}[Basic properties of arithmetic distance function on quasi-projective schemes]\label{prop:arithdistbasicqp}
Let $X$ be a quasi-projective scheme over $\KK$ and suppose $ \Delta_{X} \cap (X\times X) = \emptyset$.
Fix an arithmetic distance function $ \delta_{X}$.

\begin{enumerate}
\item {\rm (Symmetry)} We have
\begin{align*}
\delta_{X}(x,y,v) = \delta_{X}(y,x,v) + O( \lambda_{\partial (X\times X)}(x,y,v))+ O_{M_{K}}(1)
\end{align*}
for all $(x,y,v) \in (X\times X )(\KK) \times M(\KK)$. 

\item {\rm (Triangle inequality I)} We have
\begin{align*}
 \min\{  \delta_{X}(x,y,v), \delta_{X}(y,z,v)  \} \leq  \delta_{X}(x,z,v) + O( \lambda_{\partial (X \times X \times X)}(x,y,z,v))+ O_{M_{K}}(1)
\end{align*}
for all $(x,y,z,v) \in X(\KK)^{3} \times M(\KK)$.

\item {\rm (Triangle inequality II)} Let $Y \subset X$ be a closed subscheme such that $Y\cap \Ass(X) =  \emptyset$.
Fix a local height $ \lambda_{Y}$ associated with $Y$.
Then we have
\begin{align*}
\min\{ \lambda_{Y}(x,v), \delta_{X}(x,y, v) \} \leq \lambda_{Y}(y,v) + O( \lambda_{\partial (X \times X)}(x,y,v))+ O_{M_{K}}(1)
\end{align*}
for all $(x,y,v) \in X(\KK)^{2} \times M(\KK)$.

\item Let $y \in X(\KK)$. Suppose $y \in X$, as a closed point, is not an associated point of $X$.
Fix a local height function $ \lambda_{y}$ associated with $\{y\}$.
Then 
\begin{align*}
\delta_{X}(x,y, v)= \lambda_{y}(x,v) + O( \lambda_{\partial X}(x,v))+ O_{M_{K}}(1)
\end{align*}
for all $(x,v)\in (X(\KK)\setminus \{y\})\times M(\KK)$.
 
\end{enumerate}

\end{prop}
\begin{proof}
(1) Apply \cref{thm:basiclochtqp}(\ref{propertypullbackqp}) to the automorphism
\[
\sigma \colon X\times X \longrightarrow X\times X; (x,y) \mapsto (y,x).
\]
Note that $ \sigma^{-1}( \Delta_{X}) = \Delta_{X}$.

(2) Consider the following projections:
\[
\xymatrix{
& X\times X \times X \ar[dl]_{\pr_{12}} \ar[d]_{\pr_{13}} \ar[rd]^{\pr_{23}} &\\
X\times X & X\times X & X\times X.
}
\]
Note that $\pr_{ij}$ are flat and hence we have $\pr_{ij}(\Ass(X\times X \times X)) \subset \Ass(X\times X)$.
Then up to $M_{K}$-bounded functions,  we have
\begin{align*}
&\min\{ \delta_{X}(x,y,v), \delta_{X}(y,z,v)\} \\[3mm]
&= \min\{ \lambda_{ \Delta_{X}}(x,y,v) + O( \lambda_{\partial (X\times X)}(x,y,v)), \lambda_{\Delta_{X}}(y,z,v)+ O( \lambda_{\partial (X\times X)}(y,z,v))\}\\[3mm]
&=\min\{ \lambda_{\pr_{12}^{-1}( \Delta_{X})}(x,y,z,v)  , \lambda_{\pr_{23}^{-1}( \Delta_{X})}(x,y,z,v)  \}  + O( \lambda_{\partial (X\times X \times X)}(x,y,z,v))\\[1mm]
& 	\qquad  \qquad  \qquad \qquad  \qquad  \qquad  \qquad \qquad  \qquad  \qquad  \text{by \cref{thm:basiclochtqp} (\ref{propertypullbackqp})}  \\[3mm]
&= \lambda_{\pr_{12}^{-1}( \Delta_{X})\cap \pr_{23}^{-1}( \Delta_{X})}(x,y,z,v) + O( \lambda_{\partial (X\times X \times X)}(x,y,z,v)) \\[1mm]
&	\qquad  \qquad  \qquad \qquad  \qquad  \qquad  \qquad \qquad  \qquad  \qquad    \text{by \cref{thm:basiclochtqp} (\ref{propertyintersecqp})}  \\[3mm]
&\leq \lambda_{\pr_{13}^{-1}( \Delta_{X})}(x,y,z,v) + O( \lambda_{\partial (X\times X \times X)}(x,y,z,v))\\[1mm]
&\qquad \qquad  \qquad  \qquad \qquad  \qquad  \qquad \qquad \qquad  \txt{ $\pr_{12}^{-1}( \Delta_{X})\cap \pr_{23}^{-1}( \Delta_{X}) \subset \pr_{13}^{-1}( \Delta_{X})$ \\
	and \cref{thm:basiclochtqp}(\ref{propertycontainqp})} \\[3mm]
& = \lambda_{ \Delta_{X}}(x,z,v) + O( \lambda_{\partial (X\times X \times X)}(x,y,z,v))\\[3mm]
& = \delta_{X}(x,z,v) + O( \lambda_{\partial (X\times X \times X)}(x,y,z,v)).
\end{align*}

(3)
Consider the following diagram:
\[
\xymatrix{
X\times X \ar[d]_{\pr_{1}} \\
\qquad X \supset Y.
}
\]
Since $\pr_{1}$ is flat, we have $\pr_{1}(\Ass(X\times X)) \subset \Ass(X)$.
Then up to $M_{K}$-bounded functions, we have
\begin{align*}
&\min\{ \lambda_{Y}(x,v) , \delta_{X}(x,y,v)\}\\[3mm]
& = \min\{ \lambda_{\pr_{1}^{-1}(Y)}(x,y,v), \lambda_{ \Delta_{X}}(x,y,v)\} + O( \lambda_{\partial (X\times X)}(x,y,v))\\[3mm]
&= \lambda_{\pr_{1}^{-1}(Y)\cap \Delta_{X}}(x,y,v) + O( \lambda_{\partial (X\times X)}(x,y,v))
	\qquad \qquad  \text{by \cref{thm:basiclochtqp} (\ref{propertyintersecqp})} \\[3mm]
&\leq \lambda_{\pr_{2}^{-1}(Y)} (x,y,v) + O( \lambda_{\partial (X\times X)}(x,y,v))
	\qquad \qquad   \txt{ $\pr_{1}^{-1}(Y) \cap \Delta_{X} \subset \pr_{2}^{-1}(Y)$\\
	 and  \cref{thm:basiclochtqp}(\ref{propertycontainqp})}\\[3mm]
& = \lambda_{Y}(y,v) + O( \lambda_{\partial (X\times X)}(x,y,v)).
\end{align*}

(4)
Consider the embedding 
\[
i \colon X \longrightarrow X \times X; x \mapsto (x,y).
\]
Since $y \notin \Ass(X)$, $i(\Ass(X))\cap \Delta_{X} =  \emptyset$.
Thus by  \cref{thm:basiclochtqp} (\ref{propertypullbackqp}), up to $M_{K}$-bounded functions, we have
\begin{align*}
\delta_{X}(x,y,v) &= \lambda_{ \Delta_{X}}(x,y,v) + O( \lambda_{\partial (X \times X)}(x,y,v)) \\
&= \lambda_{i^{-1}( \Delta_{X})}(x,v) + O( \lambda_{\partial (X \times X)}(x,y,v)) \\
&= \lambda_{y}(x,v) + O( \lambda_{\partial (X \times X)}(x,y,v)).
\end{align*}
Here 
$  \lambda_{\partial (X \times X)}(x,y,v) = \lambda_{\partial (X \times X)} \circ (i \times \id)(x,v) = O( \lambda_{\partial X}(x,v))$
and we are done.

\end{proof}

\section{Global heights}

In this section we fix an infinite field $K$ equipped with a proper set of absolute values $M_{K}$.
We fix an algebraic closure $ \KK$ of $K$.
Recall that for an intermediate field $K \subset L \subset \KK$,
$M(L)$ is the set of absolute values on $L$ which extend elements of $M_{K}$.

\begin{cons}[Global height function associated with a presentation]\label{cons:globalhtpre}
Let $K \subset L \subset \KK$ be an intermediate field such that $[L:K] < \infty$.
Let $X$ be a projective scheme over $L$ and $Y \subset X$ be a closed subscheme such that
$Y \cap \Ass(X) =  \emptyset$.  
Let $\Y$ be a presentation of $Y$:
\begin{align*}
&\Y = (Y_{L} ; \D_{1}, \dots , \D_{r}) \qquad \text{where}  \\
&\D_{i} = ( s_{D_{i}} ; \L^{(i)}, s^{(i)}_{0}, \dots, s^{(i)}_{n_{i}} ; \M^{(i)}, t^{(i)}_{0}, \dots , t^{(i)}_{m_{i}}  ) .
\end{align*}
For any intermediate field $L \subset L' \subset \KK$, define
\begin{align*}
\lambda_{\Y, L'}  \colon (X\setminus Y)(L') \times M(L') \longrightarrow \R
\end{align*}
by 
\begin{align*}
\lambda_{\Y, L'}(x,v) = \min\{ \lambda_{\D_{1}, L'}(x,v) , \dots , \lambda_{\D_{r}, L'}(x,v)\} 
\end{align*}
where 
\begin{align*}
\lambda_{\D_{i}, L'}(x,v) =  \log \max_{0 \leq k \leq n_{i}} \min_{0\leq l \leq m_{i}}  \left\{ \left| \frac{s_{k}^{(i)}}{s_{D_{i}}t_{l}^{(i)}}(x)\right|_{v} \right\}
\end{align*}
for $(x,v) \in  (X\setminus Y)(L') \times M(L')$.
Note that this map is the induced map in \cref{prop:goodchoice}.

Now define 
\begin{align*}
h_{\Y} \colon (X\setminus Y)(\KK) \longrightarrow \R
\end{align*}
as follows.
For a point $x \in (X\setminus Y)(\KK) $, take an intermediate field $L \subset L' \subset \KK$
such that $[L' : L] < \infty$ and $x \in (X \setminus Y)(L')$.
Then 
\begin{align*}
h_{\Y}(x) = \frac{1}{[L':K]} \sum_{v \in M(L')} [L'_{v}: K_{v|_{K}}] \lambda_{\Y, L'}(x,v).
\end{align*}

\begin{claim}
This map is well-defined, that is, $h_{\Y}(x)$ is independent of the choice of $L'$.
\end{claim}
\begin{proof}
Let $L'$ and $L''$ be two intermediate fields which are finite over $L$ and such that
$x \in  (X \setminus Y)(L')$ and $x \in  (X \setminus Y)(L'')$.
We prove $h_{\Y}(x)$'s defined by $L'$ and $L''$ are the same.
By replacing $L''$ with the composite field $L' \cdot L''$, we may assume $L'\subset L''$.
Then we calculate
\begin{align*}
&\frac{1}{[L'':K]} \sum_{w \in M(L'')} [L''_{w} : K_{w|_{K}}] \lambda_{\Y, L''}(x, w)\\[3mm]
&= \frac{1}{[L'':K]} \sum_{v\in M(L')} \sum_{\tiny \txt{$w \in M(L'')$\\$w|_{L'}=v$}} [L''_{w} : L'_{v}][L'_{v} : K_{v|_{K}}] \lambda_{\Y, L''}(x,w)\\[3mm]
&=\frac{1}{[L'':K]} \sum_{v\in M(L')} \sum_{\tiny \txt{$w \in M(L'')$\\$w|_{L'}=v$}} [L''_{w} : L'_{v}][L'_{v} : K_{v|_{K}}] \lambda_{\Y, L'}(x,v)\\[-5mm]
& \qquad \qquad \qquad \qquad \qquad \qquad \qquad \qquad \qquad \qquad  \quad \text{since $x \in  (X \setminus Y)(L')$}\\[3mm]
&= \frac{[L'': L']}{[L'':K]} \sum_{v \in M(L')} [L'_{v} : K_{v|_{K}}] \lambda_{\Y, L'}(x,v) \qquad \text{by well-behavedness}\\[3mm]
&= \frac{1}{[L' : K]} \sum_{v \in M(L')} [L'_{v} : K_{v|_{K}}] \lambda_{\Y, L'}(x,v) 
\end{align*}
and we are done.
\end{proof}

\end{cons}

\begin{cons}[Global height function associated with a closed subscheme]\label{cons:globalht}

Let $X$ be a projective scheme over $\KK$.
Let $Y\subset X$ be a closed subscheme such that $Y \cap \Ass(X)=  \emptyset$.
Take an intermediate field $K \subset L \subset \KK$ such that 
\begin{itemize}
\item $[L:K] < \infty$;
\item there is a projective scheme $X_{L}$ and a closed subscheme $Y_{L} \subset X_{L}$ such that 
$(X_{L})_{\KK} \simeq X$ as $\KK$-schemes and $(Y_{L})_{\KK} \simeq Y$ via this isomorphism:
\[
\xymatrix{
(X_{L})_{\KK} \ar[r]^{\sim}  & X \\
(Y_{L})_{\KK} \ar@{^{(}->}[u] \ar[r]^{\sim} & Y \ar@{^{(}->}[u].
}
\]
\end{itemize}
We identify $(X_{L})_{\KK}$ with $X$ by this fixed isomorphism.
Note that $Y_{L} \cap \Ass(X_{L}) =  \emptyset$.
Take a presentation $\Y$ of $Y_{L}$.
The class of $h_{\Y}$ in the set of maps from $(X\setminus Y)(\KK)$ to $\R$ modulo 
bounded functions is called the global height function associated with $Y$.
Each element of the equivalence class is called a global height function associated with $Y$.
\begin{claim}
This is well-defined, that is, if $L'$ and $\Y'$ are another such choices, then
$h_{\Y}-h_{\Y'}$ is a bounded function.
\end{claim}
\begin{proof}
Let $L'' $ be a sufficiently large finite extension of the composite field $L\cdot L'$ such that
the fixed isomorphism $(X_{L})_{\KK} \simeq X \simeq (X_{L'})_{\KK}$ is defined over $L''$.
Then $\Y_{L''}$ and $\Y'_{L''}$ are presentations of $(Y_{L})_{L''} \simeq (Y_{L'})_{L''}$ (the isomorphism via the fixed isomorphism
$(X_{L})_{L''} \simeq (X_{L'})_{L''}$).
Also, we have $h_{\Y_{L''}} = h_{\Y}$ and $h_{\Y'_{L''}} = h_{\Y'}$.
Thus we may assume $L=L'$ and $\Y$ and $\Y'$ are two presentations of $Y_{L} \subset X_{L}$.

By \cref{prop:indpre}, there is an $M_{K}$-constant $ \gamma$ such that
\begin{align*}
| \lambda_{\Y}(x,v)- \lambda_{\Y'}(x,v)| \leq \gamma(v)
\end{align*}
for $(x,v) \in (X\setminus Y)(\KK) \times M(\KK)$.

Take any $x \in (X \setminus Y)(\KK) \times M(\KK)$.
Take any intermediate field $L \subset L' \subset \KK$ such that $[L': L] < \infty$ and $x \in (X_{L}\setminus Y_{L})(L')$.
Then
\begin{align*}
&h_{\Y}(x) - h_{\Y'}(x) \\[3mm]
&= \frac{1}{[L':K]} \sum_{v \in M(L')} [L'_{v}: K_{v|_{K}}] \Bigl( \lambda_{\Y, L'}(x,v) - \lambda_{\Y',L'}(x,v) \Bigr).
\end{align*}
For $w \in M(\KK)$,  set $v=w|_{L'}$ and $v_{0} = w|_{K} (= v|_{K})$.
Then by the definitions of local heights, we have
\begin{align*}
& \lambda_{\Y, L'}(x,v) = \lambda_{\Y}(x, w);\\
& \lambda_{\Y', L'}(x,v) = \lambda_{\Y'}(x, w)
\end{align*}
and therefore
\begin{align*}
|\lambda_{\Y, L'}(x,v) - \lambda_{\Y',L'}(x,v)| \leq \gamma(v_{0}).
\end{align*}
Thus 
\begin{align*}
|h_{\Y}(x) - h_{\Y'}(x) | &\leq  \frac{1}{[L':K]} \sum_{v \in M(L')} [L'_{v}: K_{v|_{K}}] \gamma(v|_{K})\\[3mm]
& = \frac{1}{[L':K]} \sum_{v_{0} \in M_{K}} \gamma(v_{0}) \sum_{\tiny \txt{$v \in M(L')$ \\ $v|_{K}=v_{0}$}} [L'_{v} : K_{v_{0}}]\\[3mm]
& = \sum_{v_{0} \in M_{K}} \gamma(v_{0})
\end{align*}
and we are done.
\end{proof}
\end{cons}

\begin{defn}
Notation as in \cref{cons:globalht}.
The equivalence class of $h_{\Y}$ modulo bounded functions is denoted by $h_{Y}$
 and called the global height function associated with $Y$.
 Each representative function is usually denoted by $h_{Y}$ too and called a global height function associated with $Y$.
\end{defn}

\begin{prop}[Basic properties of global height functions]\label{prop:basicglobalht}
Let $X$ be a projective scheme over $\KK$.
Let $Y, Z \subset X$ be closed subschemes such that $Y \cap \Ass(X) = Z \cap \Ass(X) =  \emptyset$.
Then the following hold.
\begin{enumerate}
\item $h_{Y\cap Z} \leq \min\{ h_{Y}, h_{Z} \} + O(1)$ on $(X \setminus Y\cap Z)(\KK)$. 
\item $h_{Y+Z} = h_{Y}+h_{Z} + O(1)$ on $(X \setminus (Y\cup Z))(\KK)$.
\item If $Y \subset Z$, then $h_{Y} \leq h_{Z}+O(1)$ on $(X \setminus Y)(\KK)$.
\item $\max\{ h_{Y}, h_{Z} \} \leq h_{Y\cup Z} + O(1) \leq h_{Y}+h_{Z} + O(1)$ on $(X \setminus (Y\cup Z))(\KK)$.
\item If $\Supp Y \subset \Supp Z$, then there is a positive constant $c>0$ such that
\[
h_{Y} \leq c h_{Z} + O(1)
\]
on $(X \setminus Y)(\KK)$.
\item Let $X'$ be a projective scheme over $\KK$ and $\varphi \colon X ' \longrightarrow X$ be a morphism over $\KK$.
Suppose $\varphi(\Ass(X')) \cap Y =  \emptyset$.
Then 
\[
h_{Y} \circ \varphi = h_{\varphi^{-1}(Y)} + O(1)
\]
on $(X'\setminus \varphi^{-1}(Y))(\KK)$.
\end{enumerate}
\end{prop}
\begin{proof}
These follow from the definition of global height, \cref{prop:basicprop}, and \cref{thm:basiclocht}.
\end{proof}


\begin{thebibliography}{99}


\bibitem{am} Atiyah, M. F., Macdonald, I. G.
	\textit{Introduction to commutative algebra},
	Addison-Wesley Publishing Co., Reading, Mass.-London-Don Mills, Ont. 1969 ix+128 pp. 
	
	

\bibitem{bg} Bombieri, E., Gubler, W.,
	\textit{Heights in Diophantine geometry},
	Cambridge university press, 2007.


\bibitem{hs} Hindry, M., Silverman, J. H., 
	\textit{Diophantine geometry. An introduction},
	Graduate Text in Mathematics, no.\ 20,
	Springer-Verlag, New York, 2000.


\bibitem{Lan}
Lang, S., 
	\textit{Fundamentals of Diophantine Geometry}, 
	Springer-Verlag, New York, 1983. 

\bibitem{liu} Liu, Q.
	\textit{Algebraic geometry and arithmetic curves},
	Translated from the French by Reinie Ern\'e. Oxford Graduate Texts in Mathematics, 6. Oxford Science Publications. Oxford University Press, Oxford, 2002. 	
	xvi+576 pp. ISBN: 0-19-850284-2


\bibitem{jacob} Jacobson, N.,
	\textit{Basic algebra. II},
	Second edition. W. H. Freeman and Company, New York, 1989. xviii+686 pp. ISBN: 0-7167-1933-9
	

	
\bibitem{sil87} Silverman, J.\ H.,
	\textit{Arithmetic distance functions and height functions in Diophantine geometry},
	Math. Ann. 279 (1987), no. 2, 193--216. 

	
\end{thebibliography}
\end{document}